\documentclass[10pt,a4paper]{article}
\usepackage[utf8]{inputenc}
\usepackage{amsmath}
\usepackage{amsfonts}
\usepackage{amssymb}
\usepackage{amsthm}
\usepackage{color}
\usepackage{marvosym}
\usepackage{tikz}
\definecolor{rot}{rgb}{0.75 , 0.00 , 0.00}
\definecolor{blau}{rgb}{0.00 , 0.00 , 0.5}
\usepackage[colorlinks=true, citecolor=rot, linkcolor=blau, urlcolor=blue]{hyperref}
\usetikzlibrary{arrows}
\usepackage{tikz-3dplot}
\theoremstyle{definition}
\newtheorem{defn}{Definition}
\theoremstyle{plain}
\newtheorem{thm}[defn]{Theorem}
\newtheorem{prop}[defn]{Proposition}
\newtheorem{lmm}[defn]{Lemma}
\newtheorem{cor}[defn]{Corollary}
\theoremstyle{remark}
\newtheorem*{rmk}{Remark}

\newcommand{\vertiii}[1]{{\left\vert\kern-0.25ex\left\vert\kern-0.25ex\left\vert #1
    \right\vert\kern-0.25ex\right\vert\kern-0.25ex\right\vert}}

\DeclareMathOperator{\BUC}{BUC}
\DeclareMathOperator{\Id}{Id}

\DeclareMathOperator{\VO}{VO}
\DeclareMathOperator{\VMO}{VMO}

\DeclareMathOperator{\BDO}{BDO}

\title{Fredholmness of Toeplitz operators on the Fock space}
\author{Robert Fulsche, Raffael Hagger}
\begin{document}
\maketitle
\renewcommand{\theenumi}{(\roman{enumi})}
\begin{abstract}
The Fredholm property of Toeplitz operators on the $p$-Fock spaces $F_\alpha^p$ on $\mathbb{C}^n$ is studied. A general Fredholm criterion for arbitrary operators from the Toeplitz algebra $\mathcal{T}_{p,\alpha}$ on $F_\alpha^p$ in terms of the invertibility of limit operators is derived. This paper is based on previous work, which establishes corresponding results on the unit balls $\mathbb{B}^n$ \cite{Hagger}.


\end{abstract}

\section{Introduction}
Consider the weighted Gaussian measure $d\mu_\nu(z) = (\nu/\pi)^n e^{-\nu|z|^2}dz$ on $\mathbb{C}^n$, where $\nu > 0$ and $dz$ denotes the usual Lebesgue measure on $\mathbb{C}^n \cong \mathbb{R}^{2n}$. The Fock space $F_\alpha^p$ for $\alpha > 0$, $1 < p < \infty$ is the closed subspace of $L^p(\mathbb{C}^n, \mu_{p\alpha/2})$ consisting of entire functions. Toeplitz operators on these spaces are defined to be the composition of a multiplication operator $M_f$, where $f \in L^\infty(\mathbb{C}^n)$, and a certain projection $P_\alpha$ (see below) back onto the closed subspace $F_\alpha^p$, i.e.
\begin{align*}
T_f = P_\alpha M_f.
\end{align*}
The function $f$ is then called the symbol of $T_f$ or $M_f$, respectively.

When studying such Toeplitz operators a natural property to consider is the Fredholmness of such operators, that is: are the kernel and the cokernel of $T_f$ finite dimensional? Inspired by Toeplitz operators on other spaces, e.g.~on the Hardy space or on the Bergman space over the unit ball, one expects that the information about the Fredholmness of $T_f$ can be extracted from the symbol, more precisely from the behaviour of $f$ near infinity \cite{Bottcher_Silbermann, Upmeier}. For symbols which extend continuously to the boundary sphere of $\mathbb{C}^n$ theorems of the following form are known: If $f$ is nowhere zero on the boundary sphere, then $T_f$ is Fredholm \cite{Bottcher_Wolf_1994}\cite[Theorem 2.1]{Bottcher_Wolf_1997}. Results for a more general class of symbols are known for the case $p=2$: If $f$ is of vanishing oscillation, then $T_f$ is Fredholm if $f$ is bounded away from zero close to the boundary (cf.~\cite{Perala_Virtanen, Xia_Zheng} for the Bergman space or \cite{Berger_Coburn, Stroethoff} for the Fock space). Corresponding results can be proven without the restriction $p=2$, cf.~Section 6 in this article. For Toeplitz operators with more general symbols, or even other operators from the Toeplitz algebra $\mathcal{T}_{p,\alpha}$, which is just the norm closure of the algebra generated by Toeplitz operators, such results were missing.

Only recently the methods of limit operators, known from the theory of band-dominated operators on sequence spaces, were adapted to the case of Toeplitz operators on Bergman- and Fock spaces. It was realized that the notion of ``boundary of $\mathbb{C}^n$ (resp. $\mathbb{B}^n$ in the Bergman space case)" in the usual sense was too restrictive. Instead, one densely embeds $\mathbb{C}^n$ (resp. $\mathbb{B}^n$) into the maximal ideal space $\mathcal{M}$ of $\BUC(\mathbb{C}^n)$ (resp.~$\BUC(\mathbb{B}^n)$), the space of bounded uniformly continuous functions. We thus consider $\mathcal{M} \setminus \mathbb{C}^n$ as the boundary of $\mathbb{C}^n$ and the boundary values of an operator $A \in \mathcal{T}_{p,\alpha}$ at $\mathcal{M} \setminus \mathbb{C}^n$ are obtained by ``shifting" $A$ to the boundary (we will make this precise below). For each $x \in \mathcal{M} \setminus \mathbb{C}^n$ we will get a boundary operator $A_x$, called a limit operator. In \cite{Mitkovski_Suarez_Wick} and \cite{Suarez} a limit operator theory for the Bergman space over the unit ball was developed, whereas the corresponding results for the Fock space were derived in \cite{Bauer_Isralowitz}. In both cases it was shown that operators in the Toeplitz algebra are compact if and only if all of their limit operators vanish. In \cite{Hagger} some ideas from the limit operator theory on sequence spaces were adapted to show that an operator in the Toeplitz algebra over the unit ball is Fredholm if and only if all of its limit operators are invertible. Let us briefly review a few of these ideas and where they originate. The main objective of study in limit operator theory is the class of band-dominated operators, which was first studied in its entirety by Simonenko \cite{Simonenko_1, Simonenko_2}. Specific classes of band-dominated operators were also considered earlier, e.g.~by Gohberg and Krein \cite{Gohberg_Krein}. In 1985 it was shown by Lange and Rabinovich \cite{Lange_Rabinovich} (see also \cite{Rabinovich_Roch_Silbermann_1, Rabinovich_Roch_Silbermann_2}) that a band-dominated operator is Fredholm if and only if all of its limit operators are invertible and their inverses are uniformly bounded. For the subsequent 25 years it was unclear whether the uniform boundedness condition is necessary or not. Various subclasses of operators have been studied by many different authors and in every case it was shown that the uniform boundedness condition is actually redundant, i.e.~the inverses are automatically uniformly bounded if all limit operators are invertible. For the general case this was then shown by Lindner and Seidel \cite{Lindner_Seidel} in 2014. The corresponding result for (essential) norms can be found in \cite{Hagger_Lindner_Seidel}. For a more detailed history of limit operators on sequence spaces we refer to \cite{Lindner} and \cite{Rabinovich_Roch_Silbermann_2}.

In this paper, we further adapt these results to the Fock space to obtain the following main theorem:
\begin{thm} \label{thm1}
An operator $A \in \mathcal{T}_{p,\alpha}$ is Fredholm if and only if all of its limit operators are invertible.
\end{thm}

Theorem \ref{thm1} improves an earlier result by Bauer and Isralowitz (\cite[Theorem 7.2]{Bauer_Isralowitz}) for $p =2$, where the inverses of the limit operators are additionally assumed to be uniformly bounded. Hence, besides generalizing to arbitrary $p$, we show that the uniform boundedness condition is redundant. We closely follow the lines of \cite{Hagger} in this paper. Since the Fock space is in some aspects simpler than the Bergman space, we can avoid some technical difficulties and can focus more on the actual ideas.

The paper is organized as follows: In Section 2 we introduce our notation and recall some basic results. In Section 3 we introduce band-dominated operators and provide some properties of them. Section 4 will be devoted to the theory of limit operators and the main theorem of this paper. In Section 5 methods similar to those from Section 4 will be sketched to derive results on the essential norm of operators from the Toeplitz algebra. In the end, Section 6 will be used to show how the expected results on the Fredholmness of Toeplitz operators with symbols of vanishing oscillation can be derived from our main theorem.

\section{Notation and basic definitions}
In this section we present the main definitions and some basic results which are well-known and/or easy to prove.

For $\nu > 0$ let $d\mu_{\nu}$ denote the Gaussian measure
\begin{align*}
d\mu_{\nu}(z) = \Big ( \frac{\nu}{\pi} \Big )^n e^{-\nu |z|^2} dz
\end{align*}
on $\mathbb{C}^n$, where $dz$ denotes the Lebesgue measure on $\mathbb{C}^n \simeq \mathbb{R}^{2n}$ and $| \cdot |$ denotes the norm coming from the standard hermitian inner product $\langle \cdot, \cdot \rangle$ on $\mathbb{C}^n$, which is linear in the first and antilinear in the second component. $d\mu_\nu$ is easily seen to be a probability measure. The space $L_\alpha^p$ is given by
\begin{align*}
L_\alpha^p = \{ f: \mathbb{C}^n \to \mathbb{C}; f \text{ measurable and } \| f \|_{p,\alpha} < \infty \} = L^p (\mathbb{C}^n, d\mu_{p\alpha/2})
\end{align*}
for $\alpha > 0$ and $1 < p < \infty$, where
\begin{align*}
\| f\|_{p,\alpha}^p := \int_{\mathbb{C}^n} |f(z)|^p d\mu_{p\alpha/2}(z).
\end{align*} 
Further, $F_\alpha^p$ denotes the closed subspace of entire functions in $L_\alpha^p$. Throughout this paper we will assume, unless stated otherwise, $p \in (1, \infty)$ and $\alpha \in (0,\infty)$ without further mentioning it.

For a Banach space $X$ we denote by $\mathcal{L}(X)$ the space of bounded linear operators on $X$ and by $\mathcal{K}(X)$ the ideal of compact operators. By $M_f$ we will denote the operator of multiplication by the function $f \in L^\infty(\mathbb{C}^n)$. We will use this symbol for both multiplication operators acting as $L_\alpha^p \to L_\alpha^p$ or $F_\alpha^p \to L_\alpha^p$ without mentioning $p$ or $\alpha$ in the notation. A Toeplitz operator is an operator of the form $P_\alpha M_f: F_\alpha^p \to F_\alpha^p$, where $f$ is called the symbol of the operator. Here $P_\alpha$ is the projection $L_\alpha^p \to F_\alpha^p$ onto the closed subspace given by the formula in Proposition \ref{propzhu1} below. By $\mathcal{T}_{p,\alpha}$ we denote the norm-closed subalgebra of $\mathcal{L}(F_\alpha^p)$ generated by all Toeplitz operators with bounded symbols. A net of bounded linear operators $(A_\gamma)_\gamma$ on some Banach space $X$ is said to converge $\ast$-strongly to $A \in \mathcal{L}(X)$ if $A_\gamma \to A$ strongly and $A_\gamma^\ast \to A^\ast$ strongly, where $B^\ast$ denotes the Banach space adjoint of $B \in \mathcal{L}(X)$. For a set $M \subseteq \mathbb{C}^n$ we denote its characteristic function by $\chi_M$. By $B(z, r)$ we will denote the Euclidean ball around $z \in \mathbb{C}^n$ with radius $r > 0$.

We will need the following result regarding projections from $L_\alpha^p$ to $F_\alpha^p$:

\begin{prop}[{\cite[Theorem 7.1]{Janson_Peetre_Rochberg}}]\label{propzhu1}
The linear operator $P_\alpha : L_\alpha^p \to L_\alpha^p$ given by
\begin{align*}
(P_\alpha f)(z) = \int_{\mathbb{C}^n} e^{\alpha \langle z, w\rangle}f(w) d\mu_{\alpha}(w)
\end{align*}
is a bounded projection onto $F_\alpha^p$. In particular, $P_\alpha|_{F_\alpha^p} = \Id$.
\end{prop}

The following duality results will also be of importance:
\begin{prop}\label{propzhu2}
Let $\frac{1}{p} + \frac{1}{q} = 1$. Then the following assertions hold under the usual dual pairing induced by the scalar product on $L_\alpha^2$:
\begin{enumerate}
\item $(L_\alpha^p)'\cong L_\alpha^q$,
\item for $g \in L_\alpha^q$ it holds
\[\| \langle \cdot, g\rangle_{\alpha} \|_{(L_\alpha^p)'} = \frac{2^n}{p^{n/p}q^{n/q}} \| g\|_{q,\alpha},\]
\item $(P_\alpha: L_\alpha^p \to L_\alpha^p)^\ast \cong (P_\alpha: L_\alpha^q \to L_\alpha^q)$.
\end{enumerate}
\end{prop}
\begin{proof}
The standard proof of $L^p \cong L^q$ yields (i) and (ii). Using the symmetry of $P_\alpha$ (see Proposition \ref{propzhu1}), one immediately gets (iii). We refer to \cite{Janson_Peetre_Rochberg,Zhu} for details.
\end{proof}

Similarly, we have the following duality of Fock spaces:

\begin{prop}\label{dualityfock}
Let $\frac{1}{p} + \frac{1}{q} = 1$. Then the following assertions hold under the usual dual pairing induced by the scalar product on $F_\alpha^2$:
\begin{enumerate}
\item $(F_\alpha^p)' \cong F_\alpha^q$,
\item for $g \in F_\alpha^q$ it holds 
\[ \| g\|_{q,\alpha} \leq \| \langle \cdot , g\rangle_\alpha\|_{(F_\alpha^p)'} \leq \frac{2^n}{p^{n/p}q^{n/q}} \| g\|_{q,\alpha}. \]
\end{enumerate}
\end{prop}
\begin{proof}
(i) is again standard (cf. \cite{Janson_Peetre_Rochberg,Zhu}), (ii) is \cite[Theorem 1.2]{Gryc_Kemp}. 
\end{proof}

Note that in both cases the isomorphism is not isometric for $p \neq 2$. However, these quasi-isometries still allow for the usual adjoint arguments, which we will use occasionally.

\begin{prop}\label{prop4}
Let $(U_j)_{j \in \mathbb{N}}$ be a sequence of measurable subsets of $\mathbb{C}^n$ such that every $z \in \mathbb{C}^n$ belongs to at most $N$ of the sets $U_j$ for some $N \in \mathbb{N}$. Further, let $(f_j)_{j \in \mathbb{N}}$ be a sequence of measurable functions $f_j: \mathbb{C}^n \to \mathbb{C}$ such that $\operatorname{supp} f_j \subseteq U_j$ and $|f_j(z)|\leq 1$ for all $z \in \mathbb{C}^n$. Then, for every $g \in L_a^p$
\begin{align*}
\sum_{j=1}^\infty \int_{\mathbb{C}^n} |f_j(z)g(z)|^p d\mu_{p\alpha/2} (z) \leq N \| g\|_{p,\alpha}^p.
\end{align*}
In particular,
\begin{align*}
\sum_{j=1}^\infty \| M_{f_j} g\|_{p,\alpha}^p \leq N \| g\|_{p,\alpha}^p.
\end{align*}
\end{prop}
\begin{proof}
As in \cite[Proposition 5]{Hagger}.
\end{proof}
The following two results are well-known and are provided here for completeness:
\begin{lmm}\label{lemmaboundedness}
Let $\alpha,\beta,\gamma > 0$. The function
\begin{align*}
(z,w) \mapsto e^{\alpha \langle z,w\rangle - \beta |z|^2 - \gamma |w|^2}
\end{align*}
is bounded on $\mathbb{C}^n \times \mathbb{C}^n$ if and only if $4\beta \gamma - \alpha^2 \geq 0$.
\end{lmm}
\begin{proof}
For
\begin{align*}
|e^{\alpha \langle z,w\rangle - \beta |z|^2 - \gamma |w|^2}| = e^{\alpha \operatorname{Re}\langle z,w\rangle - \beta |z|^2 - \gamma |w|^2}
\end{align*}
to be bounded it suffices to show that $\alpha \operatorname{Re}\langle z,w\rangle - \beta |z|^2 - \gamma |w|^2$ is bounded from above. Since
\begin{align*}
\alpha \operatorname{Re}\langle z,w\rangle - \beta |z|^2 - \gamma |w|^2 &\leq \alpha |z| |w| - \beta |z|^2 - \gamma |w|^2
\end{align*}
and the right-hand side of this inequality is just the polynomial $p(x,y) = \alpha xy - \beta x^2 - \gamma y^2$ evaluated at $x = |z|, y=|w|$, the boundedness follows from the well-known fact that $p(x,y)$ is bounded from above for $4\beta \gamma -\alpha^2 \geq 0$.

Conversely, if $4\beta \gamma - \alpha^2 < 0$, set $w = \sqrt{\frac{\beta}{\gamma}}z$ to obtain
\begin{align*}
e^{\alpha \langle z,w\rangle - \beta |z|^2 - \gamma |w|^2} = e^{|z|^2 ( \alpha \sqrt{\frac{\beta}{\gamma}} - 2\beta)}.
\end{align*}
Since $\alpha > 2\sqrt{\beta \gamma}$, this function is unbounded.
\end{proof}
\begin{prop}\label{cpt}
Let $D \subset \mathbb{C}^n$ be compact. Then the operators $P_\alpha M_{\chi_D}$ and $M_{\chi_D} P_\alpha: L_\alpha^p \to L_\alpha^p$ are compact operators.
\end{prop}
\begin{proof}
It is
\begin{align*}
P_\alpha M_{\chi_D} (f)(z) &= \Big (\frac{\alpha}{\pi} \Big )^n \int_{\mathbb{C}^n} e^{\alpha \langle z, w\rangle} \chi_D(w) e^{-\alpha |w|^2} f(w)dw\\
&= \Big ( \frac{2}{p} \Big )^n \int_{\mathbb{C}^n} e^{\alpha \langle z,w\rangle - \frac{2\alpha - p\alpha}{2}|w|^2} \chi_D(w)f(w) d\mu_{p\alpha/2}(w).
\end{align*}
By the Hille-Tamarkin theorem \cite[Theorem 41.6]{Zaanen} it suffices to check (recall: $L_\alpha^p = L^p(\mathbb{C}^n, \mu_{p\alpha/2})$) that
\begin{align*}
\int_{\mathbb{C}^n} \Big ( \int_{\mathbb{C}^n} |e^{\alpha \langle z,w\rangle - \frac{2\alpha - p\alpha}{2}|w|^2}|^q \chi_D(w) d\mu_{p\alpha/2}(z)\Big )^{\frac{p}{q}} d\mu_{p\alpha/2}(w) < \infty,
\end{align*}
where $\frac{1}{p} + \frac{1}{q} = 1$. A direct computation shows
\begin{align*}
\int_{\mathbb{C}^n} &\left( \int_{\mathbb{C}^n} |e^{\alpha \langle z,w\rangle - \frac{2\alpha - p\alpha}{2}|w|^2}|^q \chi_D(w) d\mu_{p\alpha/2}(z)\right)^{\frac{p}{q}} d\mu_{p\alpha/2}(w)\\
&= \Big (\frac{p\alpha}{2\pi}\Big )^{np} \int_{D} \left( \int_{\mathbb{C}^n} e^{q\alpha \operatorname{Re}\langle z,w\rangle - \frac{2q\alpha - pq\alpha}{2}|w|^2 -\frac{p\alpha}{2}|z|^2} dz \right)^{\frac{p}{q}} e^{-\frac{p\alpha}{2}|w|^2} dw \\
&= \Big (\frac{p\alpha}{2\pi}\Big )^{np}\int_D \left( \int_{\mathbb{C}^n} e^{q\alpha \operatorname{Re}\langle z,w\rangle - \frac{p\alpha}{4}|z|^2 - \frac{q^2\alpha}{p}|w|^2} e^{-\frac{p\alpha}{4}|z|^2} dz\right.\\
&\left. \quad \quad \quad \quad \quad \quad \quad \quad \quad \quad \cdot e^{\frac{-2pq\alpha + p^2q\alpha + 2q^2\alpha}{2p}|w|^2} \right)^{\frac{p}{q}} e^{-\frac{p\alpha}{2}|w|^2} dw\\
&= \Big (\frac{p\alpha}{2\pi}\Big )^{np}\int_D \left( \int_{\mathbb{C}^n} e^{q\alpha \operatorname{Re}\langle z,w\rangle - \frac{p\alpha}{4}|z|^2 - \frac{q^2\alpha}{p}|w|^2} e^{-\frac{p\alpha}{4}|z|^2} dz\right)^{\frac{p}{q}}\\
&\quad \quad \quad \quad \quad \quad \quad \quad \quad \quad \cdot e^{\frac{-3p\alpha + p^2\alpha + 2q\alpha}{2}|w|^2} dw\\\\
&\leq \Big (\frac{p\alpha}{2\pi}\Big )^{np} \int_D \left(\int_{\mathbb{C}^n} C e^{-\frac{p\alpha}{4}|z|^2} dz \right)^{\frac{q}{p}} e^{\frac{-3p\alpha + p^2\alpha + 2q\alpha}{2}|w|^2} dw\\
&< \infty,
\end{align*}
where $C$ is the bound from Lemma \ref{lemmaboundedness} with $\gamma = \frac{q^2\alpha}{p}$. The proof for $M_{\chi_D} P_\alpha$ is similar.
\end{proof}
The next lemma is essentially a sloppy version of Jensen's inequality. Since we won't need the stronger form and the notation will be more convenient with this variant, we will just mention this weak estimate. It can easily be shown and will frequently be used.
\begin{lmm}\label{lemmasum}
For $k \in \mathbb{N}$, $p \in (1, \infty)$ and $x_1, \dots, x_k \geq 0$ it is
\begin{align*}
\Big ( \sum_{j=1}^k x_j \Big )^p \leq k^p \sum_{j=1}^k x_j^p.
\end{align*}
\end{lmm}

\section{Band-dominated operators}

The aim of this section is to introduce band-dominated operators in $\mathcal{L}(L_\alpha^p)$ and to provide some basic properties of them. Most of the proofs are similar to those in \cite{Hagger} and use techniques adapted from the sequence space case (see \cite{Lindner, Rabinovich_Roch_Silbermann_2} and the references therein).

\begin{defn} \begin{enumerate}
\item An operator $A \in \mathcal{L}(L_\alpha^p)$ is called a band operator if there is a positive real number $\omega$ such that $M_f A M_g = 0$ for all $f, g \in L^\infty(\mathbb{C}^n)$ with $\operatorname{dist}(\operatorname{supp} f, \operatorname{supp} g) > \omega$. The infimum over all such $\omega$ will be denoted by $\omega(A)$ and is called the band width of $A$.
\item An operator $A \in \mathcal{L}(L_\alpha^p)$ is called a band-dominated operator if it is the norm limit of a sequence of band operators. The set of band-dominated operators on $L_\alpha^p$ will be denoted by $\BDO_\alpha^p$.
\end{enumerate}
\end{defn}

Denote by $|z|_\infty$ the induced sup-norm from $\mathbb{R}^{2n} \cong \mathbb{C}^n$ and $\operatorname{dist}_\infty (z,B) = \inf\{ |z - w|_\infty; w \in B\}$ for $z \in \mathbb{C}^n$ and $B \subseteq \mathbb{C}^n$. Set
\begin{align*}
\zeta = \{ [-3,3)^{2n} + \sigma \subset \mathbb{R}^{2n}; \sigma \in 6\mathbb{Z}^{2n} \}
\end{align*}
and enumerate $\zeta$ as $\zeta = \{ B_j\}_{j=1}^\infty$ such that it is $0 \in B_1$.
Furthermore, we denote
\begin{align*}
\Omega_k(B_j) := \{ z \in \mathbb{C}^n; \operatorname{dist}_\infty (z,B_j) \leq k\}
\end{align*}
for $k=1,2,3$ and $j \in \mathbb{N}$.
\begin{prop}[{\cite[Lemma 3.1]{Bauer_Isralowitz}}]\label{prop6}
The $B_j$ satisfy
\begin{enumerate}
\item $B_j \cap B_k = \varnothing$ for $j \neq k$;
\item every $z \in \mathbb{C}^n$ belongs to at most $2^{2n}$ of the sets $\Omega_1(B_j)$ and at most $4^{2n}$ of the sets $\Omega_3(B_j)$;
\item $\operatorname{diam}(B_j) = 6\sqrt{2n}$, where $\operatorname{diam}(B_j)$ denotes the Euclidean diameter of $B_j$.
\end{enumerate}
\end{prop}
We will now construct a sequence of auxiliary functions, which will give a partition of unity of $\mathbb{C}^n$ with particularly nice properties. Define the function $\phi: \mathbb{R} \to [0,1]$ as in Figure 1. 
\begin{figure}
\begin{center}
\begin{tikzpicture}
\draw[-] (-2.5,0) to (2.5,0);
\draw[dashed, gray] (-2.5,1) to (2.5,1);
\draw[-] (-2,0) to (-1,1);
\draw[-] (-1,1) to (1,1);
\draw[-] (1,1) to (2,0);
\draw[-] (2,2pt) to (2,-2pt);
\draw[-] (1,2pt) to (1,-2pt);
\draw[-] (-1,2pt) to (-1,-2pt);
\draw[-] (-2,2pt) to (-2,-2pt);
\draw[-] (0,2pt) to (0,-2pt);
\node[anchor=east] at (-2.5,1) {1};
\node[anchor=east] at (-2.5,0) {0};
\node[anchor=north] at (-2,-1pt) {-4};
\node[anchor=north] at (-1,-1pt) {-2};
\node[anchor=north] at (0,-1pt) {0};
\node[anchor=north] at (1,-1pt) {2};
\node[anchor=north] at (2,-1pt) {4};
\end{tikzpicture}
\end{center}
\caption{The function $\phi$}
\end{figure}
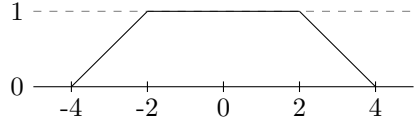
Further, define $\widetilde{\varphi}_{0}: \mathbb{C}^n \simeq \mathbb{R}^{2n} \to [0,1]$ by
\begin{align*}
\widetilde{\varphi}_{0}(x_1, \dots, x_{2n}) = \phi(x_1) \cdot \phi(x_2) \cdot \hdots \cdot \phi(x_{2n}).
\end{align*}
Now, for every $\sigma = (\sigma_1, \dots, \sigma_{2n}) \in 6\mathbb{Z}^{2n}$ set
\begin{align*}
\widetilde{\varphi}_\sigma (x_1, \dots, x_{2n}) = \widetilde{\varphi}_0(x_1 - \sigma_1, \dots, x_{2n} - \sigma_{2n}).
\end{align*}
Let $(\sigma_j)_{j \in \mathbb{N}}$ be the enumeration of $6\mathbb{Z}^{2n}$ which coincides with the enumeration of $\zeta$, i.e. $\sigma_j \in B_j$ for all $j \in \mathbb{N}$. We set $\varphi_j = \widetilde{\varphi}_{\sigma_j}$. It is easily seen that the $\varphi_j$ fulfill the following properties:
\begin{enumerate}
\item $\operatorname{supp} \varphi_j = \Omega_1 (B_j)$ for all $j$,
\item $\displaystyle \sum_{j=1}^\infty \varphi_j(z) = 1$ for all $z \in \mathbb{C}^n$,
\item the sequence $(\varphi_j)_{j \in \mathbb{N}}$ is uniformly equicontinuous (every function is even Lipschitz continuous with Lipschitz constant at most $\frac{1}{2} \cdot 2n = n$).
\end{enumerate}
In a similar way we can construct another sequence $(\psi_j)_{j \in \mathbb{N}}$ of functions, now such that the functions are non-negative and uniformly equicontinuous mappings from $\mathbb{C}^n$ to $[0,1]$ with
\begin{enumerate}
\item $\psi_j(z) = 1$ for all $z \in \Omega_2(B_j)$ for all $j \in \mathbb{N}$,
\item $\operatorname{supp} \psi_j = \Omega_3(B_j)$ for all $j \in \mathbb{N}$.
\end{enumerate}
For each $t \in (0,1)$ and each $j \in \mathbb{N}$ we define the functions $\varphi_{j,t}(z) := \varphi_j(tz)$ and $\psi_{j,t}(z) = \psi_j(tz)$. The following proposition gives a few characterisations of band-dominated operators:
\begin{prop}\label{prop8}
Let $A \in \mathcal{L}(L_\alpha^p)$. The following are equivalent:
\begin{enumerate}
\item $A$ is band-dominated;
\item $\displaystyle \lim_{t \to 0} \sup_{\| f\| = 1} \sum_{j=1}^\infty \| M_{\varphi_{j,t}} AM_{1-\psi_{j,t}} f\|^p = 0$;
\item $\displaystyle \lim_{t \to 0} \Big \| \sum_{j=1}^\infty M_{\varphi_{j,t}} AM_{1-\psi_{j,t}} \Big \| = 0$, where the convergence of the operator sum should be understood as strong convergence;
\item $\displaystyle \lim_{t \to 0} \sup_{\| f\| = 1} \sum_{j=1}^\infty \| [A, M_{\varphi_{j,t}}] f\|^p = 0$, where $[A,M_{\varphi_{j,t}}] = AM_{\varphi_{j,t}} - M_{\varphi_{j,t}}A$ is the commutator.
\end{enumerate}
\end{prop}
\begin{proof}
The strong convergence in \textit{(iii)} follows from the fact that $\sum_{j=1}^\infty M_{\varphi_{j,t}} A$ converges strongly, which can easily be seen, and the following Lemma \ref{lemma10}. \\
\textit{(i) $\implies$ (ii)}: Let $\varepsilon > 0$ and $B$ be a band operator such that $\| A - B\| < \varepsilon$. Further, let $t > 0$ be small enough such that 
\[ \operatorname{dist} \big (\operatorname{supp}(\varphi_{j,t}), \operatorname{supp}(1-\psi_{j,t})\big ) > \omega(B), \]
where the distance on the left-hand side is by construction independend of $j$. Then, for all $j \in \mathbb{N}$
\begin{align*}
M_{\varphi_{j,t}} B M_{1-\psi_{j,t}} = 0.
\end{align*}
We thus get for $f \in L_\alpha^p$
\begin{align*}
\sum_{j=1}^\infty \| M_{\varphi_{j,t}} AM_{1-\psi_{j,t}} f\|^p &= \sum_{j=1}^\infty \| M_{\varphi_{j,t}} (A - B) M_{1 - \psi_{j,t}} f\|^p\\
&\leq 2^p \sum_{j=1}^\infty ( \| M_{\varphi_{j,t}} (A - B) f\|^p \\
&\quad \quad \quad \quad + \| M_{\varphi_{j,t}} (A - B) M_{\psi_{j,t}} f\|^p)\\
&\leq 2^p 2^{2n} (\| (A - B)f\|^p + \| (A - B)M_{\psi_{j,t}}f\|^p)\\
&\leq 2^{p+1}2^{2n} \varepsilon^p \| f\|^p,
\end{align*}
where we used Proposition \ref{prop4} and Lemma \ref{lemmasum}. Since $\varepsilon > 0$ was arbitrary, the result follows.\\
\textit{(ii) $\implies$ (iii)}: Using Lemma \ref{lemmasum} combined with Proposition \ref{prop6} it is
\begin{align*}
\Big \| \sum_{j=1}^\infty M_{\varphi_{j,t}}AM_{1-\psi_{j,t}}f \Big \|^p &= \int_{\mathbb{C}^n} \Big  |\sum_{j=1}^\infty (M_{\varphi_{j,t}} AM_{1-\psi_{j,t}} f)(z)\Big |^p d\mu_{p\alpha/2}\\
&\leq \int_{\mathbb{C}^n} (2^{2n})^p \sum_{j=1}^\infty |(M_{\varphi_{j,t}} AM_{1-\psi_{j,t}}f)(z)|^p d\mu_{p\alpha/2}\\
&= (2^{2n})^p \sum_{j=1}^\infty \int_{\mathbb{C}^n} |(M_{\varphi_{j,t}} AM_{1 - \psi_{j,t}}f)(z)|^p d\mu_{p\alpha/2}\\
&= (2^{2n})^p \sum_{j=1}^\infty \| M_{\varphi_{j,t}} AM_{1-\psi_{j,t}} f\|^p.
\end{align*}
Taking the supremum over all $f$ with $\| f\|= 1$ and then the limit $t \to 0$ gives the result.\\
\textit{(iii) $\implies$ (i)}: The operator
\begin{align*}
A_m := \sum_{j=1}^\infty M_{\varphi_{j,\frac{1}{m}}} AM_{\psi_{j,\frac{1}{m}}}
\end{align*}
can easily be seen to be a band operator. Since $\sum_j \varphi_{j,t} = 1$ for all $t>0$ it is
\begin{align*}
\| A - A_m\| &= \Big \| \sum_{j=1}^\infty (M_{\varphi_{j,\frac{1}{m}}} A - M_{\varphi_{j,\frac{1}{m}}} A M_{\psi_{j,\frac{1}{m}}} ) \Big \|\\
&= \Big \| \sum_{j=1}^\infty M_{\varphi_{j,\frac{1}{m}}} AM_{1-\psi_{j,\frac{1}{m}}} \Big \|\\
&\to 0
\end{align*}
for $m \to \infty$.\\
The equivalence \textit{(i) $\Longleftrightarrow$ (iv)} is more technical and we refer to the identical proof in the unit ball case \cite[Proposition 11]{Hagger}.
\end{proof}

\begin{lmm}\label{lemma10}
For every $j \in \mathbb{N}$ let $a_j,b_j : \mathbb{C}^n \to [0,1]$ be measurable functions and assume that there is some $N \in \mathbb{N}$ such that each $z \in \mathbb{C}$ belongs to at most $N$ of the sets $\operatorname{supp}(a_j)$ and at most $M$ of the sets $\operatorname{supp}(b_j)$. If $A \in \mathcal{L}(L_\alpha^p)$, then the series
\begin{align*}
\sum_{j=1}^\infty M_{a_j}AM_{b_j}
\end{align*}
converges strongly and $\|\sum_{j=1}^\infty M_{a_j}AM_{b_j}\| \leq NM\|A\|$.
\end{lmm}
\begin{proof}
Observe that for each $f \in L_\alpha^p$ we have, as a consequence of Proposition \ref{prop4},
 $\sum_{j=m}^\infty \| M_{b_j}f\|^p \to 0$ for $m \to \infty$. To prove the lemma it suffices to show that
\begin{align*}
\Big \| \sum_{j=m}^\infty M_{a_j} AM_{b_j}f \Big \|^p \to 0
\end{align*}
for each $f \in L_\alpha^p$ as $m \to \infty$. So let $f \in L_\alpha^p$. Then
\begin{align*}
\Big \| \sum_{j=m}^\infty M_{a_j} AM_{b_j} f \Big \|^p &= \int_{\mathbb{C}^n} \Big |\sum_{j=m}^\infty (M_{a_j} AM_{b_j} f)(z)\Big |^p d\mu_{p\alpha/2}(z)\\
&\leq \int_{\mathbb{C}^n} \Big (\sum_{j=m}^\infty |(M_{a_j} AM_{b_j}f)(z)| \Big )^p d\mu_{p\alpha/2}(z).
\intertext{By assumption, the sum under the integral is pointwise a finite sum with at most $N$ terms. Using Lemma \ref{lemmasum} we can continue the estimate as follows:}
&\leq \int_{\mathbb{C}^n} N^p \sum_{j=m}^\infty |(M_{a_j} AM_{b_j}f)(z)|^p d\mu_{p\alpha/2}(z)\\
&= N^p \sum_{j=m}^\infty \int_{\mathbb{C}^n} |(M_{a_j}AM_{b_j}f)(z)|^p d\mu_{p\alpha/2}(z)\\
&= N^p \sum_{j=m}^\infty \| M_{a_j} AM_{b_j}f\|^p.
\intertext{Using $|a_j(z)|\leq 1$, it follows:}
&\leq N^p \| A\|^p \sum_{j=m}^\infty \| M_{b_j} f\|^p\\
&\to 0
\end{align*}
as $m \to \infty$. The norm estimate follows easily as well.
\end{proof}

Here are some of the properties of $\BDO_\alpha^p$:
\begin{prop}\label{prop11}
\begin{enumerate}
\item $M_f \in \BDO_\alpha^p$ for all $f \in L^\infty (\mathbb{C}^n)$.
\item $\BDO_\alpha^p$ is a closed subalgebra of $\mathcal{L}(L_\alpha^p)$.
\item If $A \in \BDO_\alpha^p$ is Fredholm and $B$ is a Fredholm regularizer of $A$, then $B \in \BDO_\alpha^p$. In particular, $\BDO_\alpha^p$ is inverse closed.
\item $\mathcal{K}(L_\alpha^p)$ is a closed and two-sided ideal in $\BDO_\alpha^p$.
\item $A \in \BDO_\alpha^p$ if and only if $A^\ast \in \BDO_\alpha^q$, where $\frac{1}{p} + \frac{1}{q} = 1$. In particular, $\BDO_\alpha^2$ is a $C^\ast$-algebra.
\end{enumerate}
\end{prop}
\begin{proof}
\textit{(i), (ii)} and \textit{(v)} are easy consequences of the definition of $\BDO_\alpha^p$. The proofs of \textit{(iii)} and \textit{(iv)} are quite technical. Since we will not need those statements for our purposes, we only refer to the identical proofs in the unit ball case in \cite[Proposition 13]{Hagger}.
\end{proof}
In the following we will show that Toeplitz operators are, in a sense made precise below, in $\BDO_\alpha^p$. The next lemma will be crucial for this.
\begin{lmm}\label{Lemma14}
There is a function $\beta_{p,\alpha}: [1, \infty) \to [0, \infty)$ with $\lim\limits_{\sigma \to \infty} \beta_{p, \alpha}(\sigma) = 0$\\ which satisfies the following property: If $(a_j)_{j \in \mathbb{N}}, (b_j)_{j \in \mathbb{N}}$ are sequences of measurable functions from $\mathbb{C}^n$ to $[0,1]$ such that
\begin{itemize}
	\item there exists $N \in \mathbb{N}$ such that each $z \in \mathbb{C}^n$ is contained in at most $N$ of the sets $\operatorname{supp}(a_j)$ and in at most $N$ of the sets $\operatorname{supp}(b_j)$,
	\item there exists $\sigma \geq 1$ such that $\operatorname{dist}\big (\operatorname{supp} a_j, \operatorname{supp}(1-b_j) \big ) \geq \sigma$ for all $j \in \mathbb{N}$,
\end{itemize}
then
\begin{align*}
\Big \| \sum_{j=1}^\infty M_{a_j} P_\alpha M_{1-b_j} \Big \| \leq N^2 \beta_{p,\alpha}(\sigma).
\end{align*}
In particular,
\begin{align*}
\Big \| \sum_{j=1}^\infty M_{a_j} P_\alpha M_{1-b_j} \Big \| \to 0
\end{align*}
for $\displaystyle \inf_{j \in\mathbb{N}} \operatorname{dist} \big ( \operatorname{supp} a_j, \operatorname{supp} (1-b_j) \big ) \to \infty$.
\end{lmm}

\begin{proof}
Observe that all operator series mentioned in the statement above and in the following proof converge $*$-strongly as an easy consequence of Lemma \ref{lemma10}. We borrow ideas from the proof of \cite[Lemma 2.6]{Bauer_Isralowitz} and sketch them here. It will be appropriate to start with the case $N=1$. We first consider the limit case $p = \infty$. Define
\begin{align*}
L_\alpha^\infty := \{ f: \mathbb{C}^n \to \mathbb{C}; f \text{ measurable and } \| f\|_{\infty,\alpha} < \infty\}
\end{align*}
with
\begin{align*}
\| f\|_{\infty,\alpha} := \underset{z \in \mathbb{C}^n}{\operatorname{ess sup}} | f(z)|e^{-\frac{\alpha}{2}|z|^2}.
\end{align*}
$P_\alpha$, that is the integral operator with the same integral kernel as for the case $p < \infty$, is a projection from $L_\alpha^\infty$ to $F_\alpha^\infty$, the closed subspace of holomorphic functions, and can hence be considered as an operator on $L_\alpha^\infty$ (see \cite[Corollary 2.22]{Zhu}). For $f \in L_\alpha^\infty$ and $z \in \mathbb{C}^n$, we get
\begin{align*}
&\Big | \sum_{j=1}^\infty (M_{a_j} P_\alpha M_{1-b_j} f)(z) \Big |e^{-\frac{\alpha}{2}|z|^2} \\
&\quad \quad \quad \leq \Big ( \frac{\alpha}{\pi} \Big )^n \sum_{j=1}^\infty |a_j(z)| \int_{\mathbb{C}^n} |1-b_j(w)| |f(w)| e^{-\frac{\alpha}{2}|w|^2} e^{-\frac{\alpha}{2}|w-z|^2} dw\\
&\quad \quad \quad \leq 2^n \| f\|_{\infty,\alpha}
\end{align*}
and hence it is $\| \sum_{j=1}^\infty M_{a_j} P_\alpha M_{1-b_j} \| \leq 2^n$ for $p = \infty$. Observe that, by the same argument with $a_1 \equiv 1 \equiv b_j$ and $a_j \equiv 0 \equiv b_1$ for all $j > 1$ we get that $P_\alpha$ is bounded on $L_\alpha^\infty$ with norm $\leq 2^n$.
If we can prove an estimate of the form $\| \sum_{j=1}^\infty M_{a_j} P_\alpha M_{1-b_j}\| \leq \beta_{p,\alpha}(\sigma)$ for $p = 2$, the result follows by interpolation for all $2 \leq p < \infty$ (see e.g. \cite[Section 9]{Janson_Peetre_Rochberg} or \cite[Chapter 2.4]{Zhu} for results on interpolation of the spaces $L_\alpha^p$). As in the proof of \cite[Lemma 2.6]{Bauer_Isralowitz}, one can prove the estimate for $p=2$ using the Schur test.

For $1 < p < 2$, instead of proving $\| \sum_{j=1}^\infty M_{a_j} P_\alpha M_{1-b_j}\| \leq \beta_{p,\alpha}(\sigma)$ directly, we will prove an estimate $\| \sum_{j=1}^\infty M_{1-b_j} P_\alpha M_{a_j}\| \leq \beta_{q,\alpha}'(\sigma)$ for the operator norm on $L_\alpha^q \cong (L^p_{\alpha})'$ and then consider adjoints (see Proposition \ref{propzhu2}). As before, the estimate on $L_\alpha^q$ can be proven with the two limit steps $q = \infty$, $q = 2$ and then using interpolation. For the case $q= \infty$, observe that in the same way as above one can show
\begin{align*}
\Big | \sum_{j=1}^\infty (M_{a_j} P_\alpha M_{b_j} f)(z) \Big |e^{-\frac{\alpha}{2}|z|^2} \leq 2^n \| f\|_{\infty,\alpha}
\end{align*}
(i.e.~replace $1-b_j$ by $b_j$) for $f \in L_\alpha^\infty$. Using
\begin{align*}
\Big \| \sum_{j=1}^\infty P_\alpha M_{a_j} f \Big \|_{\infty,\alpha} & \leq \| P_\alpha\| \Big \| \sum_{j=1}^\infty M_{a_j} f \Big \|_{\infty,\alpha} \\
&\leq \| P_\alpha\| \| f\|_{\infty,\alpha}
\end{align*}
one gets
\begin{align*}
\Big \| \sum_{j=1}^\infty M_{1-b_j} P_\alpha M_{a_j} \Big \| \leq 2^n + \| P_\alpha\|
\end{align*}
for the case $q = \infty$. For $q = 2$, we already have the estimate since
\[\left(\sum_{j=1}^\infty M_{1-b_j} P_\alpha M_{a_j}\right)^\ast = \sum_{j=1}^\infty M_{a_j} P_\alpha M_{1-b_j}\]
and the two series converge strongly.

For the case $N > 1$ we set $\Lambda_1(z) = \{ j \in \mathbb{N}; z \in \operatorname{supp}(a_j)\}$ and $\Lambda_2(z) = \{j \in \mathbb{N}; z \in \operatorname{supp}(b_j)\}$, both sets are considered to be ordered in the natural way. With 
\begin{align*}
A_j^k = \{ z \in \operatorname{supp}(a_j); j \text{ is the } \text{$k$-th element of } \Lambda_1(z)\}
\end{align*}
and 
\begin{align*}
B_j^l = \{ z \in \operatorname{supp}(b_j); j \text{ is the } \text{$l$-th element of } \Lambda_2(z)\}
\end{align*}
we have the disjoint unions $\operatorname{supp}(a_j) = A_j^1 \cup \hdots \cup A_j^N$ and $\operatorname{supp}(b_j) = B_j^1 \cup \hdots \cup B_j^N$ and both $A_{j_1}^k \cap A_{j_2}^k = \varnothing$ and $B_{j_1}^l \cap B_{j_2}^l = \varnothing$ hold for $j_1 \neq j_2$ and all $k, l = 1, \dots, N$. It follows
\begin{align*}
\sum_{j=1}^\infty M_{a_j} P_\alpha M_{1-b_j} = \sum_{k=1}^N \sum_{l=1}^N \sum_{j=1}^\infty M_{a_j \chi_{A_j^k}} P_\alpha M_{1-b_j \chi_{B_j^l}}
\end{align*}
and we can write the operator as a finite sum of operators which fulfill the requirements of the lemma for $N=1$.
\end{proof}

For an operator $A \in \mathcal{L}(F_\alpha^p)$ we define its extension to $L_\alpha^p$ by $\hat{A} = AP_\alpha + Q_\alpha$, where $Q_\alpha = \Id - P_\alpha$. Now we can prove the announced result about Toeplitz operators being band-dominated:
\begin{thm}
For any $A \in \mathcal{T}_{p,\alpha}$ it holds $\hat{A} \in \BDO_\alpha^p$.
\end{thm}

\begin{proof}
We obtain that $P_\alpha$ is in $\BDO_\alpha^p$ by combining Lemma \ref{Lemma14} and Proposition \ref{prop8} with the fact that
\begin{align*}
\lim_{t \to 0} \inf_{j \in \mathbb{N}} \operatorname{dist} \big (\operatorname{supp}(\varphi_{j,t}), \operatorname{supp}(1-\psi_{j,t}) \big ) \geq \lim_{t \to 0} \frac{2}{t} = \infty.
\end{align*}
By Proposition \ref{prop11}, the extension of every Toeplitz operator is in $\BDO_\alpha^p$ and hence, since $\BDO_\alpha^p$ is a norm-closed algebra, the result follows.
\end{proof}
The last result in this section will be a criterion about Fredholmness for band-dominated operators.

\begin{prop}\label{prop21}
Let $A \in \BDO_\alpha^p$ be such that $[A,P_\alpha] = 0$. Assume that there is a positive constant $M$ such that for every $t > 0$ there is an integer $j_0(t)> 0$ such that for all $j \geq j_0(t)$ there are operators $B_{j,t}, C_{j,t} \in \mathcal{L}(L_\alpha^p)$ with
\begin{align*}
\| B_{j,t}\|, \| C_{j,t}\| \leq M
\end{align*} 
and
\begin{align*}
B_{j,t} A M_{\psi_{j,t}} = M_{\psi_{j,t}} = M_{\psi_{j,t}} A C_{j,t}.
\end{align*}
Then $A|_{F_\alpha^p}$ is Fredholm and $\| \big (A|_{F_\alpha^p} + \mathcal{K}(F_\alpha^p) \big )^{-1} \| \leq 2^{6n+1} \| P_\alpha\| M$.
\end{prop}
\begin{proof}
The proof goes similarly to the unit ball case \cite[Proposition 17]{Hagger}. We give a sketch of the proof here: For $t>0$ define an operator
\begin{align*}
B_t := \sum_{j=j_0(t)}^\infty M_{\psi_{j,t}} B_{j,t} M_{\varphi_{j,t}},
\end{align*}
where the series converges strongly $\| B_t\| \leq 2^{6n}M$ by Lemma \ref{lemma10}. Using the identity
\begin{align*}
B_t A &= \sum_{j=j_0(t)}^\infty M_{\psi_{j,t}} B_{j,t} A M_{\varphi_{j,t}} M_{\psi_{j,t}} + \sum_{j=j_0(t)}^\infty M_{\psi_{j,t}} B_{j,t} [M_{\varphi_{j,t}},A]M_{\psi_{j,t}} \\
&\quad \quad \quad + \sum_{j = j_0(t)}^\infty M_{\psi_{j,t}}B_{j,t}M_{\varphi_{j,t}}AM_{1-\psi_{j,t}}
\end{align*}
and some properties of band-dominated operators from Proposition \ref{prop8}, one can show
\begin{align*}
\lim_{t \to 0} \Big \| B_t A - \sum_{j = j_0(t)}^\infty M_{\varphi_{j,t}} \Big \| = 0
\end{align*}
(see \cite[Proposition 17]{Hagger}). With this fact we directly obtain
\begin{align*}
\lim_{t \to 0} \Big \| P_\alpha B_t A|_{F_\alpha^p} - \sum_{j=j_0(t)}^\infty P_\alpha M_{\varphi_{j,t}}|_{F_\alpha^p} \Big \| = 0.
\end{align*}
Now since $\sum_{j=1}^{j_0(t)-1} P_\alpha M_{\varphi_{j,t}}|_{F_\alpha^p}$ is compact (Proposition \ref{cpt}) and
\[\sum_{j=1}^\infty P_\alpha M_{\varphi_{j,t}}|_{F_\alpha^p} = P_\alpha \sum_{j=1}^\infty M_{\varphi_{j,t}}|_{F_\alpha^p} =\Id,\]
$P_\alpha B_t A|_{F_\alpha^p} + \mathcal{K}(F_\alpha^p)$ converges to $\Id + \mathcal{K}(F_\alpha^p)$ in the norm of the Calkin algebra $\mathcal{L}(F_\alpha^p) / \mathcal{K}(F_\alpha^p)$ as $t\to 0$. Using a standard Neumann series argument, we obtain the existence of $B \in \mathcal{L}(F_\alpha^p)$ such that $BA|_{F_\alpha^p} \in \Id + \mathcal{K}(F_\alpha^p)$ and
\begin{align*}
\| B + \mathcal{K}(F_\alpha^p)\| \leq 2 \| P_\alpha\| \| B_t\| \leq 2^{6n+1} \| P_\alpha\| M.
\end{align*}
The other Fredholm regularizer (i.e. $A|_{F_\alpha^p} C \in \Id + \mathcal{K}(F_\alpha^p)$ for some $C \in \mathcal{L}(F_\alpha^p)$) can be obtained similarly, defining operators $C_t$:
\begin{align*}
C_t := \sum_{j=j_0(t)}^\infty M_{\varphi_{j,t}} C_{j,t} M_{\psi_{j,t}}.
\end{align*} 
Using that $A^\ast \in \BDO_\alpha^q$ one can analogously show
\begin{align*}
\lim_{t \to 0} \Big \| AC_t - \sum_{j = j_0(t)}^\infty M_{\varphi_{j,t}} \Big \| = \lim_{t \to 0} \Big \| C_t^\ast A^\ast - \sum_{j=j_0(t)}^\infty M_{\varphi_{j,t}} \Big \| = 0
\end{align*}
and conclude again, using $[A, P_\alpha] = 0$, that
\begin{align*}
\lim_{t \to 0} \Big \| AP_\alpha C_t|_{F_\alpha^p} - \sum_{j = j_0(t)}^\infty P_\alpha M_{\varphi_{j,t}}|_{F_\alpha^p} \Big \| = 0.
\end{align*}
Now proceed as in the first case.
\end{proof}

\section{Limit operators}

In this section we show our main result. As in Section 3, we use some techniques from limit operator theory on sequence spaces (see \cite{Lindner, Rabinovich_Roch_Silbermann_2}). The construction in Lemma \ref{lmm30} is due to Lindner and Seidel \cite{Lindner_Seidel}.

For each $z \in \mathbb{C}^n$ consider the weighted shift operators $C_z: L_\alpha^p \to L_\alpha^p$ given by
\begin{align*}
(C_z f)(w) = f(w - z)e^{\alpha \langle w, z\rangle - \frac{\alpha}{2} |z|^2}.
\end{align*}
$C_z$ is an isometry from $L_\alpha^p$ onto itself and from $F_\alpha^p$ onto itself for each $z \in \mathbb{C}^n$. Also, $C_z^{-1} = C_{-z}$. By $\widetilde{C}_z$ we will denote the restriction to $F_\alpha^p$. It is easy to verify that the adjoint of $\widetilde{C}_z: F_\alpha^p \to F_\alpha^p$ (in the sense of Proposition \ref{dualityfock}) is given by $\widetilde{C}_{-z}:F_\alpha^q \to F_\alpha^q$, where $q$ is the dual exponent of $p$. For an operator $A \in \mathcal{L}(F_\alpha^p)$ and $z \in \mathbb{C}^n$ we define the shifted operator $A_z$ by
\begin{align*}
A_z = \widetilde{C}_z A \widetilde{C}_{-z}.
\end{align*}
Let $\mathcal{M}$ denote the maximal ideal space of $\BUC(\mathbb{C}^n)$, the unital $C^\ast$-algebra of bounded and uniformly continuous functions on $\mathbb{C}^n$, where $\mathcal{M}$ is equipped with the weak-$^\ast$ topology. We consider $\mathbb{C}^n$ as a subset of $\mathcal{M}$ by identifying each $z \in \mathbb{C}^n$ with the functional of point evaluation at $z$, $\delta_z: f \mapsto f(z)$. In this sense, $\mathbb{C}^n$ is known to be a dense subspace of $\mathcal{M}$. If $A \in \mathcal{T}_{p,\alpha}$ and $(z_\gamma)$ is a net in $\mathbb{C}^n$ converging to $x \in \mathcal{M} \setminus \mathbb{C}^n$, then $A_{z_\gamma}$ is known to converge $\ast$-strongly to some limit operator, denoted by $A_x$, which does not depend on the particular choice of the net $(z_\gamma)$ \cite[Corollary 5.4]{Bauer_Isralowitz}.

In the following we will denote by $\tau_z: \mathbb{C}^n \to \mathbb{C}^n$ the function $w \mapsto w - z$ for each $z \in \mathbb{C}^n$. For later reference we collect a few results in the following lemma.

\begin{lmm}\label{lmm20}
\begin{enumerate}
\item For $f \in L^\infty(\mathbb{C}^n)$ and $z \in \mathbb{C}^n$ it is
\begin{align*}
C_z M_f C_{-z} = M_{f \circ \tau_z}.
\end{align*} 
\item For $z \in \mathbb{C}^n$ it is
\begin{align*}
P_\alpha C_z = \widetilde{C}_z P_\alpha.
\end{align*}
\item For $f \in L^\infty(\mathbb{C}^n)$ and $z \in \mathbb{C}^n$ it is
\begin{align*}
(T_f)_z =  T_{f \circ \tau_z}.
\end{align*}
\end{enumerate}
\end{lmm}

\begin{proof}
$(iii)$ is a direct consequence of $(i)$ and $(ii)$. For $(i)$, observe that for $g \in L_\alpha^p$ and $w \in \mathbb{C}^n$ it is
\begin{align*}
(C_{z} M_f C_{-{z}} g)(w) &= e^{\alpha \langle w,{z}\rangle - \frac{\alpha}{2}|{z}|^2} (M_f C_{-{z}} g) (w - z)\\
&= e^{\alpha \langle w, z \rangle - \frac{\alpha}{2}|z|^2} f(w - z) (C_{-z}g)(w - z)\\
&= f(w-z)g(w)
\end{align*}
and hence
\begin{align*}
C_{z} M_f C_{-z} = M_{f \circ \tau_{z}}
\end{align*}
for every $z \in \mathbb{C}^n$. $(ii)$ holds since
\begin{align*}
(P_\alpha C_z g)(w) &= \int_{\mathbb{C}^n}e^{\alpha \langle w,u\rangle} (C_z g)(u) d\mu_\alpha (u)\\
&= \Big ( \frac{\alpha}{\pi}\Big )^n\int_{\mathbb{C}^n} e^{\alpha \langle w,u\rangle + \alpha \langle u,z\rangle - \frac{\alpha}{2}|z|^2}g(u - z)e^{-\alpha |u|^2}du\\
&= \Big ( \frac{\alpha}{\pi}\Big )^n \int_{\mathbb{C}^n} e^{\alpha \langle w, v+z\rangle + \alpha \langle v+z,z\rangle - \frac{\alpha}{2}|z|^2} g(v) e^{-\alpha |v+z|^2}dv\\
&= e^{\alpha \langle w,z\rangle -\frac{\alpha}{2}|z|^2} \int_{\mathbb{C}^n} e^{\alpha \langle w-z,v\rangle}g(v) d\mu_\alpha(v)\\
&= (\widetilde{C}_z P_\alpha g)(w)
\end{align*}
for all $w \in \mathbb{C}^n$ and $g \in L^p_{\alpha}$. 
\end{proof}
\begin{prop}\label{prop22}
Let $A \in \mathcal{T}_{p,\alpha}$ and let $(z_\gamma)$ be a net in $\mathbb{C}^n$ converging to $x \in \mathcal{M} \setminus \mathbb{C}^n$ such that $A_{x}$ is invertible. Let $f \in L^\infty (\mathbb{C}^n)$ be with compact support. Then there is a $\gamma_0$ such that for all $\gamma \geq \gamma_0$ there are operators $B_\gamma, D_\gamma \in \mathcal{L}(L_\alpha^p)$ satisfying
\begin{align*}
\| B_\gamma \|, \| D_\gamma\| \leq 2(\| A_{x}^{-1}\| \|P_\alpha \| + \| Q_\alpha\|)
\end{align*}
and
\begin{align*}
B_\gamma \hat{A} M_{f \circ \tau_{-z_\gamma}} = M_{f \circ \tau_{-z_\gamma}} = M_{f \circ \tau_{-z_\gamma}} \hat{A} D_\gamma.
\end{align*}
\end{prop}
\begin{proof}
The proof is similar to the proof of \cite[Proposition 19]{Hagger}. Let $R > 0$ such that $\operatorname{supp} f \subset B(0,R)$. $P_\alpha M_{\chi_{B(0,R)}}$ is compact by Proposition \ref{cpt} and therefore it follows
\begin{align*}
\Big \| \Big ( C_{z_\gamma} (AP_\alpha + Q_\alpha)C_{-z_\gamma} &- (A_{x} P_\alpha + Q_\alpha) \Big ) M_{\chi_{B(0,R)}} \Big \|\\
&= \Big \| \Big ( C_{z_\gamma} AP_\alpha C_{-z_\gamma} - A_{x} P_\alpha) M_{\chi_{B(0,R)}} \Big \|\\
&= \Big \| (\widetilde{C}_{z_\gamma} A \widetilde{C}_{-z_\gamma} - A_{x}) P_\alpha M_{\chi_{B(0,R)}} \Big \|\\
& \to 0
\end{align*}
for $z_\gamma \to x$, where we also used $C_{z_\gamma} Q_\alpha C_{-z_\gamma} = Q_\alpha$ (which is a consequence of Lemma \ref{lmm20} $(ii)$). Therefore there exists a $\gamma_0$ such that
\begin{align*}
R_\gamma :&= (A_{x}^{-1} P_\alpha + Q_\alpha)\Big ( C_{z_\gamma} ( AP_\alpha + Q_\alpha) C_{-z_\gamma} - (A_{x} P_\alpha + Q_\alpha) \Big ) M_{\chi_{B(0,R)}}\\
&= (A_{x}^{-1} P_\alpha + Q_\alpha)C_{z_\gamma} (AP_\alpha + Q_\alpha)C_{-z_\gamma}M_{\chi_{B(0,R)}} - M_{\chi_{B(0,R)}}
\end{align*}
fulfills $\| R_\gamma\| < \frac{1}{2}$ for $\gamma \geq \gamma_0$. Here we used that $A_{x}^{-1}P_\alpha + Q_\alpha$ is the inverse of $A_{x}P_\alpha + Q_\alpha$. In particular, $\Id + R_\gamma \in \mathcal{L}(L_\alpha^p)$ is invertible for all $\gamma \geq \gamma_0$. Multiplying $R_\gamma$ by $M_f$ yields
\begin{align*}
(A_{x}^{-1} P_\alpha + Q_\alpha)C_{z_\gamma} (AP_\alpha + Q_\alpha)C_{-z_\gamma} M_f = (\Id + R_\gamma) M_f
\end{align*}
and thus
\begin{align*}
M_f = (\Id + R_\gamma)^{-1} (A_{x}^{-1} P_\alpha + Q_\alpha)C_{z_\gamma} (AP_\alpha + Q_\alpha)C_{-z_\gamma}M_f.
\end{align*}
Multiplying by $C_{-z_\gamma}$ from the left and $C_{z_\gamma}$ from the right and using Lemma \ref{lmm20} $(i)$ gives
\begin{align*}
C_{-z_\gamma}(\Id + R_\gamma)^{-1} (A_{x}^{-1} P_\alpha + Q_\alpha)C_{z_\gamma} (AP_\alpha + Q_\alpha)M_{f\circ \tau_{-z_\gamma}}= M_{f \circ \tau_{-z_\gamma}},
\end{align*}
and the claimed norm estimate follows easily with
\begin{align*}
B_\gamma := C_{-z_\gamma}(\Id + R_\gamma)^{-1}(A_{x}^{-1} P_\alpha + Q_\alpha)C_{z_\gamma}.
\end{align*}
The result for $D_\gamma$ can be derived similarly: Since $M_{\chi_{B(0,R)}} P_\alpha$ is also compact,
\begin{align*}
\Big \| M_{\chi_{B(0,R)}} \Big ( C_{z_\gamma} (A P_\alpha + Q_\alpha) C_{-z_\gamma} &- (A_{x} P_\alpha + Q_\alpha) \Big ) \Big \|\\
&= \Big \| M_{\chi_{B(0,R)}} P_\alpha ( C_{z_\gamma} A C_{-z_\gamma} - A_{x} ) P_\alpha \Big \|\\
& \to 0
\end{align*}
for $z_\gamma \to x$. Therefore
\begin{align*}
S_\gamma:= M_{\chi_{B(0,R)}} \Big ( C_{z_\gamma} (AP_\alpha + Q_\alpha) C_{-z_\gamma} - (A_{x} P_\alpha + Q_\alpha) \Big ) (A_{x}^{-1} P_\alpha + Q_\alpha)
\end{align*}
has norm $< \frac{1}{2}$ for large $\gamma$ and we get
\begin{align*}
M_{f \circ \tau_{-z_\gamma}} (AP_\alpha + Q_\alpha) D_\gamma = M_{f \circ \tau_{-z_\gamma}}
\end{align*}
with
\[ 
D_\gamma := C_{-z_\gamma} (A_{x}^{-1} P_\alpha + Q_\alpha) (\Id + S_\gamma)^{-1}C_{z_\gamma}.  \qedhere
 \]
\end{proof}
We get the following theorem:
\begin{thm}\label{thm15}
If $A \in \mathcal{T}_{p,\alpha}$ is such that $A_x$ is invertible for every $x \in \mathcal{M} \setminus \mathbb{C}^n$ and $\displaystyle \sup_{x \in \mathcal{M} \setminus \mathbb{C}^n} \| A_x^{-1} \| < \infty$, then $A$ is Fredholm.
\end{thm}
\begin{proof}
The proof works entirely as in \cite[Theorem 20]{Hagger}. For the readers convenience we reproduce it here.

Assume $A$ is not Fredholm. One easily sees that $[\hat{A}, P_\alpha] = 0$. By Proposition \ref{prop21}, there exists a strictly increasing sequence $(j_m)_{m \in \mathbb{N}}$ and some $t > 0$ with
\begin{align*}
B\hat{A}M_{\psi_{j_m,t}} \neq M_{\psi_{j_m,t}}
\end{align*}
or
\begin{align*}
M_{\psi_{j_m,t}} \hat{A}B \neq M_{\psi_{j_m,t}}
\end{align*}
for all $m \in \mathbb{N}$ and all $B \in \mathcal{L}(L_\alpha^p)$ with $\displaystyle \| B\| \leq 2\Big ( \sup_{x \in \mathcal{M} \setminus \mathbb{C}^n} \| A_x^{-1}\| \| P_\alpha\| + \| Q_\alpha\| \Big )$. Since both cases can be dealt with in the same way, we may assume
\begin{align*}
B\hat{A}M_{\psi_{j_m,t}} \neq M_{\psi_{j_m,t}}.
\end{align*}
As $\operatorname{diam}(\operatorname{supp} \psi_{j,t}) \leq \frac{12\sqrt{2n}}{t} =: R$ for all $j \in \mathbb{N}$ by definition of $\psi_{j,t}$ and Proposition \ref{prop6} $(iii)$, there is a sequence $(w_{j_m})_{m \in \mathbb{N}}$ with $|w_{j_m}| \to \infty$ such that
\begin{align*}
\operatorname{supp} \psi_{j_m,t} \subseteq B(w_{j_m}, R).
\end{align*}
By the compactness of $\mathcal{M}$ we may choose a convergent subnet $(w_\gamma)$ of $(w_{j_m})$ such that $(-w_\gamma)$ converges to some $y \in \mathcal{M} \setminus \mathbb{C}^n$. By Proposition \ref{prop22} there is a $\gamma_0$ such that for each $\gamma \geq \gamma_0$ there is an operator $B_\gamma \in \mathcal{L}(L_\alpha^p)$ with $\| B_\gamma\| \leq 2 (\| A_{y}^{-1}\| \| P_\alpha\| + \| Q_\alpha\|)$ and
\begin{align*}
B_\gamma \hat{A} M_{\chi_{B(w_\gamma, R)}} = B_\gamma \hat{A} M_{\chi_{B(0, R)} \circ \tau_{w_\gamma}} = M_{\chi_{B(0, R)} \circ \tau_{w_\gamma}} =  M_{\chi_{B(w_\gamma, R)}},
\end{align*} 
which is a contradiction.
\end{proof}

We will need the following proposition:
\begin{prop}\label{compactconv}
Let $A \in \mathcal{T}_{p,\alpha}$ be compact and $(z_\gamma)$ be a net in $\mathbb{C}^n$ converging to $x \in \mathcal{M} \setminus \mathbb{C}^n$. Then $A_{z_\gamma}$ converges $\ast$-strongly to 0.
\end{prop}
\begin{proof}
This is the statement of \cite[Theorem 1.1 and Lemma 6.1]{Bauer_Isralowitz}.
\end{proof}
The following theorem provides the converse of the previous theorem:
\begin{thm}\label{thm24}
Let $A \in \mathcal{L}(F_\alpha^p)$ be Fredholm. Let $(z_\gamma)$ be a net in $\mathbb{C}^n$ converging to $x \in \mathcal{M} \setminus \mathbb{C}^n$ such that $A_{z_\gamma}$ converges $\ast$-strongly to $A_x\in \mathcal{L}(F_\alpha^p)$. Then $A_x$ is invertible with $\| A_x^{-1}\| \leq \| \big (A + \mathcal{K}(F_\alpha^p)\big )^{-1} \|$. If further $B$ is a Fredholm regularizer of $A$, then $B_{z_\gamma}$ converges $\ast$-strongly to $A_x^{-1}$ as $z_\gamma \to x$.
\end{thm}
\begin{proof}
As in \cite[Theorem 21]{Hagger}. Since $AB - \Id$ and $BA - \Id$ are both compact, $(AB-\Id)_{z_\gamma}$ and $(BA - \Id)_{z_\gamma}$ converge $\ast$-strongly to $0$ for $z_\gamma \to x$. Further,
\begin{align*}
\| f\| &= \| \widetilde{C}_{z_\gamma} \widetilde{C}_{-z_\gamma} f\|\\
&\leq \| \widetilde{C}_{z_\gamma} BA \widetilde{C}_{-z_\gamma} f\| + \| \widetilde{C}_{z_\gamma} (I - BA)\widetilde{C}_{-z_\gamma} f\|\\
&\leq \| \widetilde{C}_{z_\gamma} B \widetilde{C}_{-z_\gamma}\| \| \widetilde{C}_{z_\gamma} A \widetilde{C}_{-z_\gamma} f\| + \| \widetilde{C}_{z_\gamma} (I - BA)\widetilde{C}_{-z_\gamma} f\|\\
&= \| B\| \| \widetilde{C}_{z_\gamma} A \widetilde{C}_{-z_\gamma} f\| + \| \widetilde{C}_{z_\gamma} (I - BA)\widetilde{C}_{-z_\gamma} f\|
\end{align*}
for all $f \in F_\alpha^p$. Letting $z_\gamma \to x$ we get $\| f\| \leq \| B\| \| A_xf\|$. $A_x$ is hence injective with closed range. Using the same argument for the adjoint operators (see Proposition \ref{dualityfock}), we get $\| g\| \leq \| B\| \| A_x^\ast g\|$ for all $g \in F_\alpha^q$ and hence the surjectivity of $A_x$. $A_x$ is therefore invertible. We also get $\| A_x^{-1}\| \leq \| B\|$ from these estimates. Since $B$ was an arbitrary Fredholm regularizer of $A$, we have $\| A_x^{-1}\| \leq \| \big ( A + \mathcal{K}(F_\alpha^p) \big )^{-1} \|$. Using the identity
\begin{align*}
B_{z_\gamma} (A_x - A_{z_\gamma}) A_{x}^{-1} + (BA - \Id)_{z_\gamma} A_x^{-1} = B_{z_\gamma} - A_x^{-1},
\end{align*}
which can easily be established, we also get the $\ast$-strong convergence of $B_{z_\gamma}$ to $A_x^{-1}$, since $A_{z_\gamma} \to A_x$, $(BA - I)_{z_\gamma} \to 0$ and $\| B_{z_\gamma}\| \leq \| B\|$. 
\end{proof}

Combining Theorem \ref{thm24} and Theorem \ref{thm15} with \cite[Corollary 5.4]{Bauer_Isralowitz}, the fact that all limit operators exist for operators in $\mathcal{T}_{p,\alpha}$, we obtain:

\begin{prop}\label{prop25a}
$A \in \mathcal{T}_{p,\alpha}$ is Fredholm if and only if $A_x$ is invertible for all $x \in \mathcal{M} \setminus \mathbb{C}^n$ and $\displaystyle \sup_{x \in \mathcal{M} \setminus \mathbb{C}^n} \| A_x^{-1}\| < \infty$.
\end{prop}
The condition $\sup \| A_x^{-1}\| <\infty$ is actually redundant. This will be shown in the remaining part of this section. Denote
\begin{align*}
r_t := \operatorname{diam} (\operatorname{supp} \varphi_{j,t}) = \frac{8\sqrt{2n}}{t},
\end{align*}
which, of course, is independent of $j$, and define for $t>0, F \subseteq \mathbb{C}^n$ and $A \in \mathcal{L}(L_\alpha^p)$
\begin{align*}
\nu(A|_F) := \inf \{ \|Af\|; f \in L_\alpha^p, \| f\|=1, \operatorname{supp} f \subseteq F\}
\end{align*}
and
\begin{align*}
\nu_t (A|_F) := \inf_{w \in \mathbb{C}^n} \nu(A|_{F \cap B(w,r_t)}).
\end{align*}
We also use the notation $\nu(A) := \nu(A|_{\mathbb{C}^n})$.
\begin{prop}\label{prop25}
For $A,B \in \mathcal{L}(L_\alpha^p)$ and $F\subseteq \mathbb{C}^n$ it is 
\begin{enumerate}
\item $|\nu(A|_F) - \nu(B|_F)|\leq \| (A-B)M_{\chi_F}\| \leq \| A-B\|$,
\item $|\nu_t(A|_F) - \nu_t(B|_F)| \leq \| A-B\|$.
\end{enumerate}
\end{prop}
\begin{proof}
\textit{(i)}: As in \cite[Proposition 27]{Hagger}: For the first statement let $\varepsilon > 0$ and pick $f \in L_\alpha^p$ with $\| f\| = 1$, $\operatorname{supp} f \subseteq F$ and $\| Bf\| \leq \nu(B|_F) + \varepsilon$. Then
\begin{align*}
\nu(A|_F) - \nu(B|_F) - \varepsilon &\leq \nu(A|_F) - \| Bf\|\\
&\leq \| Af\| - \| Bf\|\\
&\leq \| (A-B)f\|\\
&\leq \| (A-B)M_{\chi_F}\|.
\end{align*}
Since the inequalities are symmetric in $A$ and $B$, the result follows.\\
\textit{(ii)}: Let again $\varepsilon > 0$, pick $w \in \mathbb{C}^n$ such that
\begin{align*}
\nu(B|_{F \cap B(w, r_t)}) \leq \nu_t(B|_F) + \varepsilon.
\end{align*}
Then
\begin{align*}
\nu_t(A|_F) - \nu_t(B|_F) - 2\varepsilon &\leq \nu_t(A|_F) - \nu(B|_{F\cap B(w,r_t)}) - \varepsilon\\
&\leq \nu(A|_{F \cap B(w,r_t)}) - \nu(B|_{F\cap B(w,r_t)}) - \varepsilon\\
&\leq \| (A-B)M_{\chi_{F \cap B(w,r_t)}}\|\\
&\leq \| A-B\|,
\end{align*}
where the second-to-last estimate can be concluded as in the first statement. Now use again the symmetry in $A$ and $B$.
\end{proof}
\begin{prop}\label{prop26}
Let $A \in \mathcal{T}_{p,\alpha}$. For every $\varepsilon > 0$ there exists some $t > 0$ such that for all $F \subseteq \mathbb{C}^n$ and all $B \in \{ \hat{A}\} \cup \{ \hat{A}_x; x \in \mathcal{M} \setminus \mathbb{C}^n \}$:
\begin{align*}
\nu(B|_F) \leq \nu_{t}(B|_F) \leq \nu(B|_F) + \varepsilon.
\end{align*}
\end{prop}
\begin{proof}
The first inequality follows by definition. For the second inequality: Let $(A_m)_{m \in \mathbb{N}}$ be a sequence of band operators that converges to $\hat{A}$ in norm. Further, let $\varepsilon > 0 $ and choose $m \in \mathbb{N}$ such that $\| \hat{A} - A_m\| < \frac{\varepsilon}{4}$. For $x \in \mathcal{M} \setminus \mathbb{C}^n$ let $(z_\gamma)$ be a net in $\mathbb{C}^n$ converging to $x$. $(A_m)_{z_\gamma}$ is a bounded net in $\mathcal{L}(L_\alpha^p)$, hence we may pass to a weakly convergent subnet, which we also denote by $(A_m)_{z_\gamma}$. Let the limit of this net be denoted by $(A_m)_x$. The strong convergence of $\hat{A}_{z_\gamma}$ to $\hat{A}_x$ implies that $C_{z_\gamma} (\hat{A} - A_m) C_{-z_\gamma}$ converges weakly to $\hat{A}_x - (A_m)_x$. Thus,
\begin{align*}
\| \hat{A}_x - (A_m)_x\| \leq \sup_{\gamma} \| C_{z_\gamma} (\hat{A} - A_m) C_{-z_\gamma}\| = \| \hat{A} - A_m\| < \frac{\varepsilon}{4}.
\end{align*}
Now let $f, g \in L^\infty(\mathbb{C}^n)$ be such that $\operatorname{dist}(\operatorname{supp} f, \operatorname{supp}g) > \omega(A_m)$. Then Lemma \ref{lmm20} gives
\begin{align*}
M_f (C_{z_\gamma} A_m C_{-z_\gamma})M_g = C_{z_\gamma} M_{f \circ \tau_{-z_\gamma}} A_m M_{g \circ \tau_{-z_\gamma}} C_{-z_\gamma} = 0,
\end{align*}
since $\operatorname{dist}(\operatorname{supp} f \circ \tau_{-z_\gamma}, \operatorname{supp}g \circ \tau_{-z_\gamma}) = \operatorname{dist}(\operatorname{supp} f, \operatorname{supp}g)
$. This implies $\omega( (A_m)_{z_\gamma}) \leq \omega (A_m)$ and hence $\omega ((A_m)_x) \leq \omega (A_m)$ by passing to the limit. Observe now that, if we know that there exists a $t \in (0,1)$ such that for all $F \subseteq \mathbb{C}^n$ and all $B \in \{ A_m\} \cup \{ (A_m)_x; x \in \mathcal{M} \setminus \mathbb{C}^n\}$ it is
\begin{align*}
\nu_t(B|_F) \leq \nu(B|_F) + \frac{\varepsilon}{2},
\end{align*}
we are done, since by Proposition \ref{prop25} it is
\begin{align*}
|\nu(\hat{A}|_F) - \nu(A_m|_F)| \leq \| \hat{A} - A_m\| < \frac{\varepsilon}{4}
\end{align*}
and
\begin{align*}
|\nu(\hat{A}_x|_F) - \nu( (A_m)_x|_F)| \leq \| \hat{A}_x - (A_m)_x\| < \frac{\varepsilon}{4}.
\end{align*}
For the existence of such a $t$, we refer to the corresponding part of the proof of the unit ball case in \cite[Proposition 23]{Hagger}, which is identical to the situation in the Fock space.
\end{proof}
\begin{prop}\label{prop27}
$\{ A_x; x \in \mathcal{M}\}$ and $\{ A_x; x \in \mathcal{M} \setminus \mathbb{C}^n\}$ are both compact in the strong operator topology for each $A \in \mathcal{T}_{p,\alpha}$.
\end{prop}
\begin{proof}
$\mathcal{M}$ and $\mathcal{M} \setminus \mathbb{C}^n$ are compact and $x \mapsto A_x$ is continuous w.r.t. the strong operator topology \cite[Proposition 5.3]{Bauer_Isralowitz}.
\end{proof}
\begin{lmm}\label{lmm28}
Let $A \in \mathcal{T}_{p,\alpha}$, $w \in \mathbb{C}^n$ and $r > 0$. Then, for each $f \in L_\alpha^p$ with $\operatorname{supp} f \subseteq B(w, r)$ and every $x \in \mathcal{M} \setminus \mathbb{C}^n$ there exists $g \in L_\alpha^p$ and $y \in \mathcal{M} \setminus \mathbb{C}^n$ with $\| g\| = \| f\|$, $\operatorname{supp} g \subseteq B(0, r)$ and $\| \hat{A}_x f\| = \| \hat{A}_{y} g\|$. Further, $\nu(\hat{A}_{y}|_{B(0,r + |w|)}) \leq \nu(\hat{A}_x|_{B(0,r)})$.
\end{lmm}
\begin{proof}
Using the definition, one can quickly check that
\begin{align*}
\widetilde{C}_{w_1} \widetilde{C}_{w_2} = \widetilde{C}_{w_1 + w_2} e^{\frac{\alpha}{2}(\langle w_2, w_1\rangle - \langle w_1, w_2\rangle)}
\end{align*}
for every $w_1, w_2 \in \mathbb{C}^n$. Let $(z_\gamma)$ be a net in $\mathbb{C}^n$ that converges to $x$. Taking a suitable subsequence if necessary, we may assume that
\begin{align*}
\widetilde{C}_{-w} \widetilde{C}_{z_\gamma} A \widetilde{C}_{-z_\gamma} \widetilde{C}_{w} &= \widetilde{C}_{z_\gamma - w} A \widetilde{C}_{-(z_\gamma - w)} \to A_y
\end{align*}
for some $y \in \mathcal{M} \setminus \mathbb{C}^n$ by Proposition \ref{prop27}, and hence $\widetilde{C}_{-w} A_x \widetilde{C}_w = A_{y}$. Since $P_\alpha C_w = \widetilde{C}_w P_\alpha$, we also have $C_{-w} \hat{A}_x C_{w} = \hat{A}_{y}$. Now let $f \in L_\alpha^p$ be such that $\operatorname{supp} f \subseteq B(w,r)$. Then $g := C_{-w} f$ satisfies $\| g\| = \| f\|$, $\operatorname{supp} g \subseteq B(0,r)$ and $\| \hat{A}_{y} g\| = \| \hat{A}_x f\|$.\\
For the second statement, pick a function $h \in L_\alpha^p$ such that $\operatorname{supp} h \subseteq B(0,r)$. Then $C_{-w}h$ satisfies $\operatorname{supp} (C_{-w}h) \subseteq B(-w,r) \subseteq B(0,r + |w|)$ and $\| \hat{A}_x h\| = \| \hat{A}_{y} C_{-w} h\|$.
\end{proof}
\begin{lmm}\label{lmm30}
Let $A \in \mathcal{T}_{p,\alpha}$. Then there exists a $y \in \mathcal{M} \setminus \mathbb{C}^n$ such that
\begin{align*}
\nu(\hat{A}_y) = \inf \{ \nu(\hat{A}_x); x \in \mathcal{M} \setminus \mathbb{C}^n \}.
\end{align*}
\end{lmm}
\begin{proof}
We only give a sketch of the proof here, since it is identical (up to the obvious changes) to the proof in \cite[Lemma 25]{Hagger}.\\
Using Proposition \ref{prop26} we get a sequence $(t_k)_{k \in \mathbb{N}}$ with $r_{t_{k+1}} > 2r_{t_k}$ and
\begin{align*}
\nu_{t_k} (B|_F) \leq \nu(B|_F) + \frac{1}{2^{k+1}}
\end{align*}
for all $k \in \mathbb{N}, F \subseteq \mathbb{C}^n$ and $B \in \{ \hat{A}\} \cup \{ \hat{A}_x; x \in \mathcal{M} \setminus \mathbb{C}^n\}$. Further, let $(x_j)_{j \in \mathbb{N}}$ be a sequence in $\mathcal{M} \setminus \mathbb{C}^n$ such that
\begin{align*}
\lim_{j \to \infty} \nu(\hat{A}_{x_j}) = \inf \{\nu(\hat{A}_x); \ x \in \mathcal{M} \setminus \mathbb{C}^n \}.
\end{align*}
Using Lemma \ref{lmm28} repeatedly we can construct a sequence $(y_j)_{j \in \mathbb{N}} \subseteq \mathcal{M} \setminus \mathbb{C}^n$ such that
\begin{align*}
\nu(\hat{A}_{y_j}|_{B(0, 4r_{t_k})}) \leq \nu(\hat{A}_{x_j}) + \frac{1}{2^{k-1}}.
\end{align*}
Passing to a strongly convergent subnet $(A_{y_{j_\gamma}})_\gamma$ of $(A_{y_j})_{j \in \mathbb{N}}$ (Proposition \ref{prop27}), which converges to $A_y$ for some $y \in \mathcal{M} \setminus \mathbb{C}^n$, we get
\begin{align*}
\| (\hat{A}_{y_{j_\gamma}} - \hat{A}_y)M_{\chi_{B(0, 4r_{t_k})}}\| \to 0
\end{align*}
by Proposition \ref{cpt} and hence
\begin{align*}
\nu(\hat{A}_{y_{j_\gamma}}|_{B(0, 4r_{t_k})}) \to \nu(\hat{A}_y|_{B(0, 4r_{t_k})})
\end{align*}
by Proposition \ref{prop25}. Then
\begin{align*}
\nu(\hat{A}_y) &\leq \nu(\hat{A}_y|_{B(0, 4r_{t_k})}) = \lim_\gamma \nu(\hat{A}_{y_{j_\gamma}}|_{B(0, 4r_{t_k})}) \\
&\leq \lim_\gamma \nu(\hat{A}_{x_{j_\gamma}}) + \frac{1}{2^{k-1}} = \lim_{j \to \infty} \nu(\hat{A}_{x_j}) + \frac{1}{2^{k-1}}.
\end{align*}
Taking the limit $k \to \infty$, we get the desired result.
\end{proof}
We can now finally state and prove our main result. This extends a result of Bauer and Isralowitz (\cite[Theorem 7.2]{Bauer_Isralowitz}) to arbitary $p$ and shows that the uniform boundedness condition is actually redundant. Su\'arez et al.~showed a similar result for the unit ball (\cite[Theorem 5.8]{Mitkovski_Suarez_Wick}, \cite[Theorem 10.3]{Suarez}), which was then improved in \cite{Hagger}.
\begin{thm} \label{main_thm}
For $A \in \mathcal{T}_{p,\alpha}$ the following are equivalent:
\begin{enumerate}
\item $A$ is Fredholm,
\item $A_x$ is invertible and $\| A_x^{-1}\| \leq \| \big (A + \mathcal{K}(F_\alpha^p) \big )^{-1}\|$ for all $x \in \mathcal{M}\setminus \mathbb{C}^n$,
\item $A_x$ is invertible for all $x \in \mathcal{M} \setminus \mathbb{C}^n$ and $\displaystyle \sup_{x \in \mathcal{M} \setminus \mathbb{C}^n} \| A_x^{-1}\| < \infty$,
\item $A_x$ is invertible for all $x \in \mathcal{M} \setminus \mathbb{C}^n$,
\item $\hat{A}_x$ is invertible for all $x \in \mathcal{M} \setminus \mathbb{C}^n$.
\end{enumerate}
\end{thm}
\begin{proof}
The equivalence of (i), (ii) and (iii) was already proven in Proposition \ref{prop25a}. That (iii) implies (iv) is clear. If $A_x$ is invertible, $A_x^{-1} P_\alpha + Q_\alpha$ is an inverse of $\hat{A}_x$. Therefore (iv) implies (v). We finish the proof by showing that (v) implies (iii): Using that $\nu(B) = \| B^{-1}\|^{-1} > 0$ if $B$ is invertible, we get
\begin{align*}
\sup_{x \in \mathcal{M} \setminus \mathbb{C}^n} \| \hat{A}_x^{-1}\| = \sup_{x \in \mathcal{M} \setminus \mathbb{C}^n} \frac{1}{\nu(\hat{A}_x)} = \frac{1}{\nu(\hat{A}_y)} < \infty,
\end{align*}
where $y$ is from Lemma \ref{lmm30}. An inverse for each $A_x$ is given by $\hat{A}_x^{-1}|_{F_\alpha^p}$. Since $\| B\| \leq \| \hat{B}\|$ for every $B \in \mathcal{L}(F_\alpha^p)$ we also get $\displaystyle \sup_{x \in \mathcal{M} \setminus \mathbb{C}^n} \| A_x^{-1}\| < \infty$. The fact that $\nu(B) = \| B^{-1}\|^{-1}$ for invertible $B$ can be found in \cite[Lemma 2.35]{Lindner}.
\end{proof}
We get the following corollary directly from the definition of the essential spectrum, which is defined as 
\begin{align*}
\sigma_{ess}(A) := \{ \lambda \in \mathbb{C}; \ A-\lambda \text{ is not Fredholm} \}
\end{align*}
for a bounded linear operator $A$.
\begin{cor}\label{cor31}
For each $A \in \mathcal{T}_{p,\alpha}$ it is
\begin{align*}
\sigma_{ess}(A) = \bigcup_{x \in \mathcal{M} \setminus \mathbb{C}^n} \sigma(A_x).
\end{align*}
\end{cor}
We emphasize that the redundancy of the uniform boundedness condition in Theorem \ref{main_thm} is essential for this corollary (cf.~\cite[Theorem 7.3]{Bauer_Isralowitz}, \cite[Corollary 5.9]{Mitkovski_Suarez_Wick}, \cite[Corollary 10.4]{Suarez}).

\section{Norm estimates}
The aim of this section is to provide estimates of the essential norm for operators $A \in \mathcal{T}_{p,\alpha}$. Here we adapt some ideas from \cite{Hagger} and \cite{Hagger_Lindner_Seidel} and (slightly) improve \cite[Theorem 6.2, Theorem 7.1]{Bauer_Isralowitz}.

Define for $t > 0$, $F \subseteq \mathbb{C}^n$ and $A \in \mathcal{L}(F_\alpha^p)$
\begin{align*}
\| AP_\alpha|_F\| := \sup \{ \| AP_\alpha f\|; f \in L_\alpha^p, \| f\|=1, \operatorname{supp}  f \subseteq F\}
\end{align*}
and
\begin{align*}
\vertiii{AP_\alpha|_F}_t := \sup_{w \in \mathbb{C}^n} \| AP_\alpha|_{F \cap B(w, r_t)}\|.
\end{align*}
\begin{prop}\label{prop31}
For every $A \in \mathcal{T}_{p,\alpha}$ and every $\varepsilon > 0$ there exists a $t>0$ such that for all $F \subseteq \mathbb{C}^n$ and every $B \in \{ A \} \cup \{A_x; x \in \mathcal{M} \setminus \mathbb{C}^n\}$ it is
\begin{align*}
\| BP_\alpha|_F\| \geq \vertiii{BP_\alpha|_{F}}_t \geq \| BP_\alpha|_F\| - \varepsilon.
\end{align*}
\end{prop}
\begin{proof}
The proof is very similar to the proof of Proposition \ref{prop26}. Only the second inequality needs to be proven, as the first follows directly from the definition. If $A_m$ is a band operator such that $\| AP_\alpha - A_m\| < \frac{\varepsilon}{4}$, one can prove that 
\begin{align*}
 \vertiii{C|_F}_t \geq \| C|_F \| - \varepsilon
\end{align*}
holds for all $F \subset \mathbb{C}^n$ and all $C \in \{ A_m\} \cup \{ (A_m)_x; x \in \mathcal{M} \setminus \mathbb{C}^n\}$, where $(A_m)_x$ is defined as in the proof of Proposition \ref{prop26} (cf. \cite[Proposition 27]{Hagger}). Then, using estimates similar to those in Proposition \ref{prop25}, one gets the desired result.
\end{proof}

This now allows us to give an alternative proof of the second part of \cite[Theorem 6.2]{Bauer_Isralowitz}. Our approach additionally shows that the constants there may in fact be chosen as $1$ and ${\| P_\alpha\|}^{-1}$. We do not know whether these constants are optimal, though.

\begin{thm}\label{thm27}
Let $A \in \mathcal{T}_{p,\alpha}$. Then
\begin{align*}
\frac{1}{\| P_\alpha\|} \| A + \mathcal{K}(F_\alpha^p)\| \leq \sup_{x \in \mathcal{M} \setminus \mathbb{C}^n} \| A_x\| \leq \| A + \mathcal{K}(F_\alpha^p)\|.
\end{align*}
\end{thm}
\begin{proof}
Let $(z_\gamma)$ be a net converging to $x \in \mathcal{M} \setminus \mathbb{C}^n$. $(A + K)_{z_\gamma}$ converges $\ast$-strongly to $A_x$ for every $K \in \mathcal{K}(F_\alpha^p)$ since $K_{z_\gamma}$ converges to $0$ by Proposition \ref{compactconv}. Using Banach-Steinhaus and the fact that $\widetilde{C}_w$ is an isometry for each $w \in \mathbb{C}^n$, one gets
\begin{align*}
\| A_x\| \leq \sup_{\gamma} \| \widetilde{C}_{z_\gamma} (A + K) \widetilde{C}_{-z_\gamma}\| = \| A+K\|.
\end{align*}
Since this holds for all $x \in \mathcal{M} \setminus \mathbb{C}^n$ and all compact operators $K$, the second inequality follows. We give a sketch for the proof of the first inequality, and refer to \cite[Theorem 28]{Hagger} for the missing details.\\
It can be seen that it suffices to prove
\begin{align*}
\inf_{K \in \mathcal{K}(L_\alpha^p, F_\alpha^p)} \| AP_\alpha + K\| \leq \sup_{x \in \mathcal{M} \setminus \mathbb{C}^n} \| A_x P_\alpha\|,
\end{align*}
where $\mathcal{K}(L_\alpha^p, F_\alpha^p)$ is the set of compact operators from $L_\alpha^p$ to $F_\alpha^p$. This will be proven by contradiction: Assume that
\begin{align*}
\inf_{K \in \mathcal{K}(L_\alpha^p, F_\alpha^p)} \| AP_\alpha + K\| > \sup_{x \in \mathcal{M} \setminus \mathbb{C}^n} \| A_x P_\alpha\| + \varepsilon
\end{align*}
for some $\varepsilon > 0$. By Proposition \ref{cpt},
\begin{align*}
\| AP_\alpha|_{\mathbb{C}^n \setminus B(0,s)}\| = \| AP_\alpha - AP_\alpha M_{\chi_{B(0,s)}}\| > \sup_{x \in \mathcal{M} \setminus \mathbb{C}^n} \| A_x P_\alpha\| + \varepsilon
\end{align*}
for all $s > 0$. By Proposition \ref{prop31} there is a $t \in (0,1)$ with
\begin{align*}
\vertiii{AP_\alpha|_{\mathbb{C}^n \setminus B(0,s)}}_t \geq \| AP_\alpha |_{\mathbb{C}^n \setminus B(0,s)}\| - \frac{\varepsilon}{2} > \sup_{x \in \mathcal{M} \setminus \mathbb{C}^n} \| A_x P_\alpha\| + \frac{\varepsilon}{2}.
\end{align*}
Using the definition, for each $s > 0$ there must be some $w_s \in \mathbb{C}^n$ such that
\begin{align*}
\| AP_\alpha M_{\chi_{B(w_s, r_t)}}\| \geq \| AP_\alpha M_{\chi_{B(w_s,r_t) \setminus B(0,s)}}\| > \sup_{x \in \mathcal{M} \setminus \mathbb{C}^n} \| A_x P_\alpha\| + \frac{\varepsilon}{2}.
\end{align*}
Using that $M_{\chi_{B(w_s,r_t)}} = C_{w_s} M_{\chi_{B(0,r_t)}} C_{-w_s}$, $P_\alpha C_{w_s} = \widetilde{C}_{w_s} P_\alpha$ (Lemma \ref{lmm20}) and the fact that $\widetilde{C}_{-w_s}$ and $C_{w_s}$ are surjective isometries, we get
\begin{align*}
\| A_{-w_s} P_\alpha M_{\chi_{B(0,r_t)}}\| > \sup_{x \in \mathcal{M} \setminus \mathbb{C}^n} \| A_x P_\alpha\| + \frac{\varepsilon}{2}.
\end{align*}
Since $(w_s)$ clearly cannot converge in $\mathbb{C}^n$ and $\mathcal{M}$ is compact, there is a subnet of $(w_s)$, also denoted by $(w_s)$, such that $-w_s$ converges to $y \in \mathcal{M} \setminus \mathbb{C}^n$ and $A_{-w_s}$ converges to $A_{y}$ strongly, which implies by the compactness of $P_\alpha M_{\chi_{B(0,r_t)}}$ (Proposition \ref{cpt})
\begin{align*}
\| A_{-w_s} P_\alpha M_{\chi_{B(0,r_t)}} \| \to \| A_{y} P_\alpha M_{\chi_{B(0,r_t)}}\|.
\end{align*}
But this implies
\begin{align*}
\| A_{y}P_\alpha M_{\chi_{B(0,r_t)}} \| \geq \sup_{x \in \mathcal{M} \setminus \mathbb{C}^n} \| A_x P_\alpha\| + \frac{\varepsilon}{2},
\end{align*}
which is a contradiction.
\end{proof}
An improvement of this can be obtained if $p = 2$. Comparing to \cite[Theorem 7.1]{Bauer_Isralowitz} (or \cite[Theorem 5.6]{Mitkovski_Suarez_Wick}, \cite[Theorem 10.1]{Suarez} in case of the unit ball), this shows that the supremum is actually a maximum.
\begin{thm}
For $A \in \mathcal{T}_{2,\alpha}$ it is
\begin{align*}
\| A + \mathcal{K}(F_\alpha^2)\| = \max_{x \in \mathcal{M} \setminus \mathbb{C}^n} \| A_x\|.
\end{align*}
\end{thm}
\begin{proof}
Replacing $\nu$ and $\nu_t$ by $\|\cdot \|$ and $\vertiii{\cdot}_t$ in the proof of Lemma \ref{lmm30} and using Proposition \ref{prop31}, one can show that there is a $y \in \mathcal{M} \setminus \mathbb{C}^n$ such that $\| A_y P_\alpha\| = \sup \{ \| A_xP_\alpha\|; x \in \mathcal{M} \setminus \mathbb{C}^n \}$. Since $\| A_x P_\alpha\| = \| A_x\|$ for $p=2$, we get that the supremum in the theorem is actually a maximum. The equality follows by Theorem \ref{thm27} with $\| P_\alpha\| = 1$.
\end{proof}
\section{Symbols of vanishing oscillation and vanishing mean oscillation}
For a bounded and continuous function $f: \mathbb{C}^n \to \mathbb{C}$ we define
\begin{align*}
\operatorname{Osc}_z^r(f) := \sup \{ |f(z) - f(w)|; w \in \mathbb{C}^n, |z-w|\leq r\} 
\end{align*}
for $z \in \mathbb{C}^n$ and $r>0$. A basic result about oscillations is that
\begin{align*}
\lim_{|z| \to \infty} \operatorname{Osc}_z^r(f) = 0 \text{ for all } r>0 \Leftrightarrow \lim_{|z| \to \infty} \operatorname{Osc}_z^r(f) = 0 \text{ for one } r>0.
\end{align*}
Set
\begin{align*}
\VO(\mathbb{C}^n) := \big \{ f: \mathbb{C}^n \to \mathbb{C}; f \text{ bounded and continuous}, \lim_{|z|\to \infty} \operatorname{Osc}_z^1 (f) = 0 \big \}.
\end{align*}
It is easy to see that $\VO(\mathbb{C}^n) \subset \BUC(\mathbb{C}^n)$. In the case $p=2$ it is well-known \cite{Berger_Coburn, Stroethoff} that for $\VO$-symbols the Fredholm information is located at the boundary, i.e.~for $f \in \VO(\mathbb{C}^n)$ it holds
\begin{align*}
\sigma_{ess}(T_f) = f(\partial \mathbb{C}^n),
\end{align*}
where $f(\partial \mathbb{C}^n)$ denotes the set of limit points of $f(z)$ as $|z| \to \infty$. We get those results for every $1 < p < \infty$ as a special case of Corollary \ref{cor31}:
\begin{thm} \label{thm32}
For $f \in \VO(\mathbb{C}^n)$ it holds
\begin{align*}
\sigma_{ess}(T_f) = f(\partial \mathbb{C}^n).
\end{align*}
\end{thm}
\begin{proof}
By Corollary \ref{cor31} we need to show that
\begin{align*}
\bigcup_{x \in \mathcal{M} \setminus \mathbb{C}^n} \sigma({(T_f)_x}) = f(\partial \mathbb{C}^n).
\end{align*}
Let $(z_\gamma)$ be a net in $\mathbb{C}^n$ converging to $x \in \mathcal{M} \setminus \mathbb{C}^n$. Since $f \in \BUC(\mathbb{C}^n)$ it is $f\circ \tau_{z_\gamma}(0) = f(z_\gamma) \to x(f)$. Further, observe that
\begin{align*}
|(f \circ \tau_{z_\gamma})(0) - (f \circ \tau_{z_\gamma})(w)| = |f(z_\gamma) - f(z_\gamma - w)| \leq \operatorname{Osc}_{z_\gamma}^r(f) \to 0
\end{align*}
for $|w| \leq r$. Therefore $f \circ \tau_{z_\gamma}$ converges uniformly on compact subsets to the constant function $x(f)$. Hence, using Lemma $\ref{lmm20}$,
\begin{align*}
(T_f)_{z_\gamma} = T_{f \circ \tau_{z_\gamma}} \overset{s}{\to} T_{x(f)} = (T_f)_x
\end{align*}
where $T_{x(f)}$ is just $x(f)\cdot \Id$, thus $\sigma( (T_f)_x) = \{ x(f)\}$. But since $f(z_\gamma)$ converges to $x(f)$, this needs to be in $f(\partial \mathbb{C}^n)$ and hence
\begin{align*}
\sigma_{ess} (T_f) \subseteq f(\partial \mathbb{C}^n).
\end{align*}
On the other hand, if $w \in f(\partial \mathbb{C}^n)$, let $(z_m)$ be a sequence such that $f(z_m) \to w$. Using Proposition \ref{prop27} we may choose a convergent subnet $((T_f)_{z_\gamma})_\gamma$ of $((T_f)_{z_m})_m$ and, as above, it converges to $w \Id$ and we get the other implication.
\end{proof}
We also get the following corollary for symbols with vanishing mean oscillation (see e.g.~\cite{Bauer} for a definition):
\begin{cor}
If $f \in \VMO(\mathbb{C}^n) \cap L^\infty (\mathbb{C}^n)$, then
\begin{align*}
\sigma_{ess}(T_f) = \tilde{f}(\partial \mathbb{C}^n),
\end{align*}
where $\tilde{f}$ is the Berezin transform of $f$.
\end{cor}
\begin{proof}
It is $\tilde{f} \in \VO(\mathbb{C}^n)$ \cite[Corollary 2.8]{Bauer} and also $\widetilde{|f - \tilde{f}|^2} \in C_0(\mathbb{C}^n)$, i.e. $\widetilde{|f - \tilde{f}|^2}$ vanishes at infinity \cite[Theorem 5.3]{Bauer}. This of course implies $\widetilde{f - \tilde{f}} \in C_0(\mathbb{C}^n)$. Therefore, $T_{f - \tilde{f}}$ is compact \cite[Theorem 1.1]{Bauer_Isralowitz} and hence
\[ \sigma_{ess}(T_f) = \sigma_{ess}(T_{\tilde{f}}) = \tilde{f}(\partial \mathbb{C}^n). \qedhere \]
\end{proof}

\begin{rmk}
After submitting this paper, we noticed that Theorem \ref{thm32} was independently found recently by Al-Qabani and Virtanen \cite{Al-Qabani_Virtanen} using different methods.
\end{rmk}

\bibliographystyle{amsplain}
\bibliography{References}

\bigskip

\noindent
\begin{tabular}{l l}
Robert Fulsche & Raffael Hagger\\
\href{fulsche@math.uni-hannover.de}{\Letter fulsche@math.uni-hannover.de} & \href{raffael.hagger@math.uni-hannover.de}{\Letter raffael.hagger@math.uni-hannover.de}
\end{tabular}
\\

\noindent
Both authors:\\
Institut f\"{u}r Analysis\\
Leibniz Universit\"at Hannover\\
Welfengarten 1\\
30167 Hannover\\
GERMANY
\end{document}